\documentclass[twoside,11pt]{article}
\usepackage{graphicx,amsmath,amsthm,amssymb,latexsym,amsfonts,color,cite}
\usepackage[bookmarksnumbered, colorlinks, plainpages]{hyperref}

\newtheorem{thm}{\bf Theorem}
\newtheorem{lem}{\bf Lemma}

\newtheorem{example}{\bf Example}
\newtheorem{definition}{\bf Definition}

\newtheorem{cor}{Corollary}

\newtheorem{alg}{Algorithm}
\newcommand{\norm}[1]{{ \left\| { #1 } \right\| }}
\def\h{\hspace{-0.2cm}}
\def\nul{{\rm null}}

\begin{document}
\title{\bf A modification of the generalized shift-splitting method for singular saddle point problems}
\author{{Davod Khojasteh Salkuyeh\footnote{Corresponding author} and Maryam Rahimian}\\ [2mm]
{\small \textit{Faculty of Mathematical Sciences, University of Guilan, Rasht, Iran }}\\
{\small \textit{E-mails: khojasteh@guilan.ac.ir, maryam.rahimian1988@gmail.com}\textit{}}}
\date{}
\maketitle
{\bf Abstract.} A modification of the generalized shift-splitting (GSS) method is presented for solving singular saddle point problems. In this kind of modification, the diagonal shift matrix is replaced by a block diagonal matrix which is symmetric positive definite. Semi-convergence of the proposed method is investigated. The induced preconditioner is applied to the saddle point problem and the preconditioned system is solved by the restarted generalized minimal residual method. {Eigenvalue distribution of the preconditioned matrix is also discussed.} Finally some numerical experiments are given to show the effectiveness and robustness of the new preconditioner. Numerical results show that the modified GSS method is superior to the classical GSS method.  \\[-3mm]

\pagestyle{myheadings}\markboth{D.K. Salkuyeh, M. Rahimian}{On the GSS method for singular saddle point problems}
\thispagestyle{empty}

\section{Introduction}
We consider the following saddle point problem
\begin{align}\label{eq1}
\mathcal{A}u\equiv\left(\begin{matrix}
A & B^T\\
-B & 0\\
\end{matrix}\right)\left(\begin{matrix}
x\\
y\\
\end{matrix}\right)=\left(\begin{matrix}
f\\
g\\
\end{matrix}\right)=b
\end{align}
where $A\in \mathbb{R}^{n \times n}$ is a nonsymmetric positive definite matrix ($x^TAx>0$ for all $0\neq x\in \mathbb{R}^n$) and $B\in \mathbb{R}^{m\times n}$ is rank deficient ($rank(B)=r<m\leqslant n$). In addition, we assume that the matrices $A$ and $B$ are sparse and  $f\in \mathbb{R}^n$ and  $g\in\mathbb{R}^m$ are two given vectors.  It is easy to verify that the matrix $\mathcal{A}$ is singular. We also assume that the singular saddle point problem \eqref{eq1} is consistent. Saddle point problems of the form \eqref{eq1} appear in a variety of scientific and engineering problems; e.g., computational fluid dynamics, constrained optimization \cite{IFISS}. It is mentioned that if $r=m$, then the coefficient matrix $\mathcal{A}$ is nonsingular and the saddle point problem \eqref{eq1} has a unique solution \cite{Benzi-a1}.

Several efficient iterative methods have been presented to solve the saddle point problems. Some Uzawa-type schemes have been presented to solve saddle point problems in \cite{Bai3,Bai2,Chao1,Chen2,Chen1}. Bai et al. in \cite{Bai-HSS} proposed the Hermitian and skew-Hermitian splitting (HSS) method for solving non-Hermitian positive definite system of linear equations. Next, Benzi and Golub \cite{Benzi-a1} applied the HSS iterative method to the saddle point problem and  investigated its convergence properties. They applied the induced HSS preconditioner to the saddle point problem and solved the preconditioned system by the generalized minimal residual (GMRES) method. A preconditioned HSS (PHSS) iterative method involving a single parameter was established by Bai et al. in \cite{Bai-PHSS}. Then, Bai and Golub  proposed its two-parameter acceleration, called the accelerated Hermitian and skew-Hermitian splitting (AHSS) iterative method \cite{Bai-AHSS}.

Bai et al. in \cite{Bai-JCM} presented a shift-splitting preconditioner for solving the system of linear equations $Ax=b$ where $A$ is a large and sparse non-Hermitian positive definite matrix. In fact, using the shift-splitting of the matrix $A$ as
\begin{equation}
A=M(\alpha)- N(\alpha)=\dfrac{1}{2}(\alpha I+A)-\dfrac{1}{2}(\alpha I-A),
\end{equation}
they proposed the shift-splitting iteration method
\[
M(\alpha)x^{(k+1)}=N(\alpha)x^{(k)}+b,
\]
where $\alpha>0$. This method serves the  preconditioner $M(\alpha)$, called shift-splitting preconditioner, for the system $Ax=b$. Numerical results presented  in \cite{Bai-JCM} show that the shift-splitting preconditioner induces effective preconditioned Krylov subspace iterative methods. Using the idea of \cite{Bai-JCM} and in the case that the matrix $A$ is symmetric positive definite and $B$ is of full rank, a shift-splitting preconditioner was presented by Cao et al. in \cite{Cao-CAM-2014} for the saddle point problem. In fact, for a given $\alpha>0$, they split the matrix ${\cal A}$ as
\[
\mathcal{A}=\frac{1}{2}(\alpha \mathcal{I}+\mathcal{A})-\frac{1}{2}(\alpha \mathcal{I}-\mathcal{A}),
\]
and used the matrix
\[
\mathcal{P}_{SS}=\frac{1}{2}(\alpha \mathcal{I}+\mathcal{A}),
\]
as a preconditioner for the saddle point problem. When $A$ is nonsymmetric positive definite and $B$ is of full rank, Cao et al. in \cite{Cao-AML-2015} considered the generalized shift-splitting (GSS)
\[
\mathcal{A}=\dfrac{1}{2}
\left(
  \begin{array}{cc}
\alpha I+A & B^T\\
-B & \beta I
  \end{array}
\right)-\dfrac{1}{2}
\left(
  \begin{array}{cc}
\alpha I-A & -B^T\\
B & \beta I
  \end{array}
\right),
\]
and investigated the convergence properties of the corresponding stationary iterative method
\begin{equation}\label{GSSM}
\left(
  \begin{array}{cc}
\alpha I+A & B^T\\
-B & \beta I
  \end{array}
\right)u^{(k+1)}=
\left(
  \begin{array}{cc}
\alpha I-A & -B^T\\
B & \beta I
  \end{array}
\right)u^{(k)}+2b,
\end{equation}
where $\alpha,\beta>0$. Then, they applied the matrix
\[
\mathcal{P}_{GSS}=\dfrac{1}{2}\left(\begin{matrix}
\alpha I+A & B^T\\
-B & \beta I
\end{matrix}\right),
\]
as a preconditioner for the saddle point problem. Obviously, when $\alpha=\beta$, the GSS method reduces to the shift-splitting method. In \cite{Ren}, Ren et al. investigated the eigenvalue distribution of the shift-splitting preconditioned
saddle point matrix and showed that all eigenvalues having nonzero imaginary parts are located in an intersection of two circles and all real eigenvalues are located in a positive interval. Shi et al. in \cite{Shi} provided eigenvalue bounds for the nonzero eigenvalues of the shift-splitting  preconditioned singular nonsymmetric saddle point matrices.

Salkuyeh et al. in \cite{Salkuyeh-AML-2015} applied the generalized shift-splitting preconditioner to the generalized saddle point problems with symmetric positive definite $(1,1)$-block and symmetric positive semi-definite $(2,2)$-block. Then they developed the results for the same problem when the symmetry of the matrix is omitted \cite{Salkuyeh-Conf}. Cao and Miao in \cite{Cao-CAMWA-2016} analyzed semi-convergence of the GSS method for the singular saddle point problem (\ref{eq1}). Shen and Shi in \cite{Shen-CAMWA-2016} applied the GSS method to a class of singular generalized saddle point problems and analyzed the semi-convergence properties of the method. {Recently, in \cite{Huang3}}, Huang and Huang presented a class of generalized shift-splitting (GSS) iterative methods to solve \eqref{eq1} when $A$ is positive real matrix and $B$ is rank-deficient matrix. They investigated the semi-convergence property of the GSS iterative method and gave the sharper bounds of the eigenvalues for the GSS iterative method and proposed the inexact GSS iterative method.

In this paper, a modification of the GSS iterative method (MGSS) is presented for the singular saddle point problem (\ref{eq1}). In the special cases of the MGSS method one obtains the shift-splitting and generalized shift-splitting methods.

 {Throughout the paper, for a complex number $z$, $\Re(z)$ and $\Im(z)$ are denoted for the real and the imaginary parts of $z$, respectively. For a complex matrix $A$, the conjugate transpose of $A$ is denoted by $A^*$. For a Hermitian positive definite matrix $A$, the $A$-norm of a vector $x$ is defined by $\|x\|_A=\sqrt{x^*Ax}$. For two square matrices $A$ and $B$, we write $A\succ B$ (resp. $A\succeq B$) if $A-B$ is symmetric positive definite (resp. symmetric positive semidefinite). In the same way, $A\prec B$ and $A\preceq B$ are defined. For a nonsigular matrix $C$, the spectral condition number of $C$ is denoted by $\kappa(C)$, i.e., $\kappa(C)=\|C\|_2\|C^{-1}\|_2$. For a square matrix $A$, the spectral radius of $A$ is defined by $\rho(A)$, i.e., $\rho(A)=\max_{\lambda\in \sigma(A)}|\lambda|$, where $\sigma(A)$  is the spectrum of $A$. For a matrix $A$, $\nul(A)$ is used for the null space of $A$.}

This paper is organized as follows. The MGSS iterative method is proposed in Section \ref{Sec2} and its semi-convergence properties are presented in Section \ref{Sec3}. Section \ref{Sec4} is devoted to the spectral analysis of the preconditioned matrix.  Numerical experiments are presented in \ref{Sec5}. The paper is ended by some concluding remarks in Section \ref{Sec6}.

\section{The MGSS iterative method}\label{Sec2}
 Let
\begin{equation}\label{eq-Omega}
\Omega =\left(\begin{matrix}
H & 0\\
0 & Q\\
\end{matrix}\right),\end{equation}
where  $ H\in \mathbb{R}^{n\times n}$ and  $ Q\in \mathbb{R}^{m\times m}$ are symmetric positive definite matrices. Then the matrix $\mathcal{A}$ is split as
$\mathcal{A}=\mathcal{M}-\mathcal{N}$, where
\[
\mathcal{M}=\dfrac{1}{2}(\Omega +\mathcal{A}),\quad \mathcal{N}=\dfrac{1}{2}(\Omega -\mathcal{A}).
\]
It is easy to see that the matrix $\mathcal{M}$ is nonsingular. Similar to the GSS method we consider the MGSS method as
\begin{equation}\label{EqGHSS1}
\mathcal{M}u^{(k+1)}=\mathcal{N}u^{(k)}+b,
\end{equation}
for the saddle point problem (\ref{eq1}), where $u^{(0)}$ is an initial guess. Eq. \eqref{EqGHSS1} is equivalent to
\begin{equation}\label{eq3}
\frac{1}{2}\left(\begin{matrix}
H+A & B^T\\
-B & Q\\
\end{matrix}\right)\left( \begin{array}{c}
    x^{(k+1)} \\
    y^{(k+1)}
  \end{array}\right)=\frac{1}{2}\left(\begin{matrix}
H-A & -B^T\\
B & Q\\
\end{matrix}\right)\left( \begin{array}{c}
    x^{(k)} \\
    y^{(k)}
  \end{array}\right)+\left(\begin{matrix}
f\\
g\\
\end{matrix}\right).
\end{equation}
Obviously if $H=\alpha I$ and $Q=\alpha I$, then the MGSS method reduces to the shift-splitting method, and $H=\alpha I$ and $Q=\beta I$, then the MGSS method coincides with the GSS method. In \cite{ZhengCAM}, when the matrix $A$ is symmetric positive definite the convergence of the
iterative method (\ref{eq3}) was investigated. Denoting $\Gamma=\mathcal{M}^{-1}\mathcal{N}$ and $h=\mathcal{M}^{-1}b$, the iterative method \eqref{EqGHSS1} can be rewritten as
\begin{equation}\label{MGSS}
u^{(k+1)}=\Gamma u^{(k)}+h.
\end{equation}
We have
\begin{equation}\label{eqqq}
\Gamma=\mathcal{M}^{-1}\mathcal{N}=\mathcal{I}-\mathcal{M}^{-1}\mathcal{A}.
\end{equation}
Since $\mathcal{A}$ is singular,  {this relation shows} that 1 is an eigenvalue of $\Gamma$ and as a result we have $\rho(\Gamma)\geqslant 1$, where $\rho(.)$ denotes the spectral radius of the matrix.  {Hence, we focus on} the semi-convergence of the MGSS method. To do so, we need the following definition and lemma.

\begin{definition}(see \cite{Berman:1994})\label{def1}
The iterative method \eqref{MGSS} is semi-convergent if, for any initial guess $(x_0;y_0)$, the iteration sequence $(x_k;y_k)$ produced by \eqref{MGSS} converges to a solution  $(x_\star;y_\star)$ of \eqref{eq1}. Moreover, it holds
\begin{equation}
\left[\begin{array}{c}
x_\star\\
y_\star
\end{array}\right]=(I-\Gamma)^Dc+(I-E)\left[\begin{array}{c}
x_0\\
y_0
\end{array}\right],
\end{equation}
with $E=(I-\Gamma)(I-\Gamma)^D$, where $I$ is the identity matrix and $(I-\Gamma)^D$ denotes  the Drazin inverse of $I-\Gamma$.
\end{definition}
\begin{lem} (see \cite{Berman:1994})\label{lm1}
The iterative method \eqref{MGSS} is semi-convergent if and only if the following conditions hold:\\
$\left(1\right)$ $Index(I-\Gamma)=1$, i.e., $rank(I-\Gamma)=rank(I-\Gamma)^2$; \\
$\left(2\right)$ The pseudo-spectral radius of $\Gamma$ satisfies $$\vartheta(\Gamma)=\max\lbrace\vert\lambda\vert, \lambda\in\sigma(\Gamma), \lambda\neq1\rbrace<1.$$
\end{lem}


\section{Semi-convergence of the MGSS iterative method}\label{Sec3}

In this section, we present the semi-convergence analysis of the MGSS iterative method.
Let $\lambda$ be an eigenvalue of the matrix $\mathcal{P}_{MGSS}=\mathcal{M}^{-1}\mathcal{N}$ and $u=(x;y)$ be the corresponding eigenvector.
Hence, we have $\mathcal{N}u=\lambda\mathcal{M}u$ or equivalently
\begin{equation}\label{eq4}
\begin{cases}
(H-A)x-B^Ty=\lambda(H+A)x+\lambda B^Ty,\\
Bx+Qy=-\lambda Bx+\lambda Qy.
\end{cases}
\end{equation}

\begin{lem}\label{lm2}
 Assume that $A\in \mathbb{R}^{n\times n}$ is nonsymmetric positive definite and $B\in \mathbb{R}^{m\times n}$ $(m\leqslant n)$ is rank-deficient. Let $H\in \mathbb{R}^{n\times n}$ and $Q\in \mathbb{R}^{m\times m}$ be symmetric positive definite matrices. If $\lambda$ is an eigenvalue of the matrix $\mathcal{P}_{MGSS}$, then $\lambda\neq -1$.
\end{lem}
\begin{proof}
If $\lambda=-1$, from Eq. \eqref{eq4}, we obtain $\Omega u=0$ which is a contradiction, because  $u\neq 0$ and $\Omega$ is nonsingular.
\end{proof}
\begin{lem}\label{lm3}
Let $A$ be a nonsymmetric positive definite matrix and $B$ be a rank-deficient matrix. Assume that $H$ and $Q$ are symmetric positive definite matrices.  Then, $\lambda=1$ if and only if $x=0$.
\end{lem}
\begin{proof}
If $\lambda=1$, from Eq. \eqref{eq4} we obtain
\begin{equation}\label{eq10}
\begin{cases}
Ax=-B^Ty,\\
Bx=0.
\end{cases}
\end{equation}
Multiplying both sides of the first equality of Eq. \eqref{eq10} by $x^*$, implies that $x^*Ax=-(Bx)^*y=0$ and this yields $x=0$.

If $x=0$, then from the second equality of \eqref{eq4} we obtain $Qy=\lambda Qy$, which yields $\lambda=1$ since $y\neq 0$.
\end{proof}


\begin{thm}\label{Theo1}
 Assume that $A\in \mathbb{R}^{n\times n}$ is nonsymmetric positive definite and $B\in \mathbb{R}^{m\times n}$ $(m\leqslant n)$ is rank-deficient. Also assume that $H\in \mathbb{R}^{n\times n}$ and $Q\in \mathbb{R}^{m\times m}$ are symmetric positive definite matrices. Then, $\vartheta(\Gamma)<1$.
\end{thm}
\begin{proof}
Without loss of generality let $\norm{x}_H^2=x^* Hx=1$. Multiplying both sides of the first equation in \eqref{eq4} by $x^*$ yields
\begin{equation}\label{eq6}
\dfrac{1-\lambda}{1+\lambda}=x^*Ax+(Bx)^*y.
\end{equation}
Also from the second equation in \eqref{eq4} we have
\begin{equation}\label{eq7}
Bx=\omega Qy,\quad \rm{with}~\omega=\dfrac{\lambda-1}{\lambda+1}.
\end{equation}
Substituting Eq. \eqref{eq7} in \eqref{eq6} yields
\begin{equation}
\omega =-x^*Ax-\bar{\omega}y^*Qy.
\end{equation}
Therefore, we have $\Re(\omega)=-\Re(x^*Ax)-\Re(\omega)y^*Qy$. Hence, we deduce that
\begin{equation}\label{eq8}
\Re(\omega)=-\dfrac{\Re(x^*Ax)}{1+y^*Qy}\leqslant 0.
\end{equation}
On the other hand, we have $\omega=\dfrac{\lambda-1}{\lambda+1}$, which is equivalent to
\begin{equation*}
\lambda=\dfrac{1+\omega}{1-\omega}=\dfrac{1+\Re(\omega)+i\Im(\omega)}{1-\Re(\omega)-i\Im(\omega)}.
\end{equation*}
Then
\begin{eqnarray}\label{eq9}
\vert\lambda\vert=\sqrt{\dfrac{(1+\Re(\omega))^2+(\Im(\omega))^2}{(1-\Re(\omega))^2+(\Im(\omega))^2}}.
\end{eqnarray}
From Eqs. \eqref{eq8} and \eqref{eq9}, we get $\vert\lambda\vert\leqslant 1$. To complete the proof we need to prove that if $\vert\lambda\vert=1$, then $\lambda=1$. If $\vert\lambda\vert=1$, then it follows from Eq. \eqref{eq9} that $\Re(\omega)=0$. This, together with Eq. \eqref{eq8} gives $\Re(x^*Ax)=0$. Since $A$ is positive definite, it eventuates $x=0$. Therefore, from Lemma \ref{lm3} we conclude that $\lambda=1$.
\end{proof}


\begin{thm}\label{Theo2}
Suppose that $A\in \mathbb{R}^{n\times n}$ is nonsymmetric positive definite and $B\in \mathbb{R}^{m\times n}(m\leqslant n)$ is rank-deficient. Let $H\in \mathbb{R}^{n\times n}$ and $Q\in \mathbb{R}^{m\times m}$ be symmetric positive definite matrices. Then,  $rank(I-\Gamma)=rank(I-\Gamma)^2$, where $\Gamma$ is the iteration matrix of the MGSS method.
\end{thm}
\begin{proof}
Since $\Gamma=\mathcal{M}^{-1}\mathcal{N}=I-\mathcal{M}^{-1}\mathcal{A}$, we deduce that $rank(I-\Gamma)^2=rank(I-\Gamma)$ holds if and only if $\nul(\mathcal{M}^{-1}\mathcal{A})^2=\nul(\mathcal{M}^{-1}\mathcal{A})$.
It is clear that $\nul((\mathcal{M}^{-1}\mathcal{A})^2)\supseteq \nul(\mathcal{M}^{-1}\mathcal{A})$. Hence, all we need is to show that
\begin{equation}\label{eq5}
\nul(\mathcal{M}^{-1}\mathcal{A})^2\subseteq \nul(\mathcal{M}^{-1}\mathcal{A}).
\end{equation}
Let $r=(r_1;r_2)\in \nul(\mathcal{M}^{-1}\mathcal{A})^2$ with $r_1\in \mathbb{R}^n$ and $r_2\in \mathbb{R}^m$. This means that $(\mathcal{M}^{-1}\mathcal{A})^2r=0$ which is equivalent to $\mathcal{A}\mathcal{M}^{-1}\mathcal{A}r=0$. Letting $s=(s_1;s_2)=\mathcal{M}^{-1}\mathcal{A}r$, we have $\mathcal{A}s=0$, which is equivalent to
\begin{equation}\label{eq7s}
\begin{cases}
As_1+B^Ts_2=0,\\
-Bs_1=0.
\end{cases}
\end{equation}
Multiplying the first equality of  Eq. \eqref{eq7s} by ${s_1}^T$, implies that $${s_1}^TAs_1+(Bs_1)^Ts_2=0.$$
Thus, using the second equality of \eqref{eq7s} we obtain ${s_1}^TAs_1=0$ and since $A$ is positive definite, this implies that $s_1=0$. Hence, from the first equation in \eqref{eq7s}, we deduce that $B^Ts_2=0$.  From $s=\mathcal{M}^{-1}\mathcal{A}r$, we have $\mathcal{M}s=\mathcal{A}r$ which is equivalent to
\begin{eqnarray*}
\left(\begin{array}{cc}
H+A & B^T\\
-B & Q
\end{array}\right)
\left(\begin{array}{c}
0\\
s_2
\end{array}\right)&\hspace{-0.2cm}=\hspace{-0.2cm}&2\left(\begin{array}{cc}
A & B^T\\
-B & 0
\end{array}\right)\left( \begin{array}{c}
r_1\\
r_2
\end{array}\right),
\end{eqnarray*}
which can be written as
\begin{equation}
\begin{cases}\label{eq8s}
B^Ts_2=2Ar_1+2B^Tr_2,\\
Qs_2=-2Br_1.
\end{cases}
\end{equation}
From the second equation in \eqref{eq8s} we see that  $s_2=-2Q^{-1}Br_1$. Therefore, it follows from $B^Ts_2=0$ that
\[
(Br_1)^TQ^{-1}Br_1=0,
\]
and hence $Br_1=0$. Therefore $s_2=0$ which completes the proof.
\end{proof}

According to Lemma \ref{lm1}  and  Theorems \ref{Theo1} and \ref{Theo2} the semi-convergence of the MGSS method was proved. We use the preconditioner $\mathcal{M}$ for a Krylov subspace method such as GMRES, or its restarted version GMRES($\ell$) to solve system \eqref{eq1}. We require to compute a vector of the form $z=\mathcal{M}^{-1}r$ for using the preconditioner $\mathcal{M}$ within a Krylov subspace method where $r=(r_1;r_2)$ with $r_1\in \mathbb{R}^n$ and $r_2\in \mathbb{R}^m$. By some computation, we can write
\[
\mathcal{M}=\dfrac{1}{2}\left(\begin{matrix}
I & 0\\
-B(H+A)^{-1} & I
\end{matrix}\right)\left(\begin{matrix}
H+A & 0\\
0 & S
\end{matrix}\right)\left(\begin{matrix}
I & (H+A)^{-1}B^T\\
0 & I
\end{matrix}\right),
\]
where $S=Q+B(H+A)^{-1}B^T$. Hence,
\begin{equation}\label{M-1}
\mathcal{M}^{-1}=2\left(\begin{matrix}
I & -(H+A)^{-1}B^T\\
0 & I
\end{matrix}\right)\left(\begin{matrix}
(H+A)^{-1} & 0\\
0 & S^{-1}
\end{matrix}\right)\left(\begin{matrix}
I & 0\\
B(H+A)^{-1} & I
\end{matrix}\right).
\end{equation}
By applying \eqref{M-1}, we present the following algorithm to compute the vector $z=(z_1;z_2)$ with $z_1\in \mathbb{R}^n$ and $z_2\in \mathbb{R}^m$
\begin{alg}\rm\label{Alg1}
Computation of $z=\mathcal{M}^{-1}r$. \\[-5mm]
\begin{enumerate}
\item Solve $(H+A)w=2r_1$ for $w$. \\[-6mm]
\item Compute $w_2=2r_2+Bw$. \\[-6mm]
\item Solve $Sz_2=w_2$ for $z_2$. \\[-6mm]
\item Solve  $(H+A)t=B^T z_2$ for $t$.\\[-6mm]
\item Compute $z_1=w-t$.
\end{enumerate}
\end{alg}

It is worth noting that all the presented results for the MGSS method hold when $A$ is symmetric positive definite. In this case, both of the matrices $H+A$ and $S$ are symmetric positive definite and the corresponding systems in Algorithm \ref{Alg1} can be solved exactly by the Cholesky factorization or inexactly by the conjugate gradient method. In the case that the matrix $A$ is nonsymmetric positive definite we can use the LU factorization to solve both of the systems exactly or by a Krylov subspace method like the GMRES method inexactly.

Since the matrix $S$ in Algorithm \ref{Alg1} involves the term $B(A+H)^{-1}B^T$, for large problems, solving the corresponding system using an exact solver like LU factorization may be very expensive. In general both of the matrices $A$ and $H+A$ are nonsymmetric positive definite. Hence,  to compute
vector $z_2$ in step 3 of Algorithm  \ref{Alg1} it is recommended to  use a Krylov-subspace method like the GMRES($\ell$) method. Within the GMRES($\ell$) we need to compute vectors like $y=Sx$, where $x$ is given vector. To compute $y$ we can use the following algorithm.
\begin{alg}\rm\label{Alg2}
Computation of  $y=Sx$. \\[-5mm]
\begin{enumerate}
\item $q_1:=B^Tx$\\[-6mm]
\item Solve $(H+A)q_2=q_1$ for $q_2$ using the LU factorization of $H+A$ \\[-6mm]
\item $y:=Qx+Bq_2$
\end{enumerate}
\end{alg}

\section{Eigenvalue distribution of the preconditioned matrix}\label{Sec4}

In this section, we discuss the spectral properties of the preconditioned matrix $\mathcal{M}^{-1}\mathcal{A}$. Since in the preconditioner matrix $\mathcal{M}$, the multiplicative factor $1/2$ has no effect on the preconditioned system, we drop it and use $\mathcal{K}=\Omega+\mathcal{A}$ as a preconditioner.

Let $\mu$ be an eigenvalue of the preconditioned matrix $\mathcal{K}^{-1}\mathcal{A}$. From Eq. \eqref{eqqq}, each eigenvalue $\lambda$ of $\Gamma$ can be written as $\lambda=1-2\mu$, where $\mu\in\sigma(\Gamma)$. Therefore,
\[
| 1-2\mu| =| \lambda |\leqslant \rho(\Gamma)\leqslant 1,
\]
which is equivalent to
\begin{equation}\label{eq13}
| \mu-\frac{1}{2}|\leqslant \frac{1}{2}.
\end{equation}
This shows that the eigenvalues of the preconditioned  matrix are located in a circle centered at the point $(\frac{1}{2},0)$ with radius $\frac{1}{2}$.

On the other hand, using Eq. (\ref{M-1}) the preconditioned matrix can be written as
\begin{eqnarray}
\nonumber\mathcal{K}^{-1}\mathcal{A} &\h=\h& \left(\begin{matrix}
I & -(H+A)^{-1}B^T\\
0 & I
\end{matrix}\right)\left(\begin{matrix}
(H+A)^{-1} & 0\\
0 & S^{-1}
\end{matrix}\right)\left(\begin{matrix}
I & 0\\
B(H+A)^{-1} & I
\end{matrix}\right)\left(\begin{matrix}
A & B^T\\
-B & 0
\end{matrix}\right)\\
&\h=\h& \left(\begin{matrix}
L  & (H+A)^{-1}B^T(I-K)\\
S^{-1}B((H+A)^{-1}A-I) & K
\end{matrix}\right),\label{Equa22}\qquad
\end{eqnarray}
where $L=(H+A)^{-1}(A-B^TS^{-1}B(H+A)^{-1}A+B^TS^{-1}B)$ and $K=S^{-1}B(H+A)^{-1}B^T$.

Since the matrix $H$ is unknown, it is difficult in general  to characterize the eigenvalue distribution of the preconditioned matrix.
However, we consider the special case that the matrix $H$ depends on the parameter $\alpha$ (say $H_{\alpha}$) such that $H_{\alpha}\rightarrow 0$ as $\alpha$ tends to zero. In this case, we use $\mathcal{K}_\alpha$, $L_\alpha$ and $K_\alpha$ instead of $\mathcal{K}$, $L$ and $K$, respectively.

\begin{thm}\label{Theo3}
 {Suppose that $A\in \mathbb{R}^{n\times n}$ is nonsymmetric positive definite and $B\in \mathbb{R}^{m\times n}(m\leqslant n)$ is rank-deficient. Let $H_{\alpha}\in \mathbb{R}^{n\times n}$ and $Q\in \mathbb{R}^{m\times m}$ be symmetric positive definite matrices. Then, $n$ eigenvalues of preconditioned matrix $\mathcal{K}_{\alpha}^{-1}\mathcal{A}$ tends to $1$, as $\alpha$ approaches to zero. The rest of the eigenvalues of $\mathcal{K}_{\alpha}^{-1}\mathcal{A}$ are of the form
\[
\gamma_{\alpha}=\dfrac{r_{\alpha}}{q_{\alpha}+r_{\alpha}},
\]
where $r_{\alpha}$  is a complex number with $\Re(r_{\alpha})\geqslant 0$ and  $q_{\alpha}>0$.}
\begin{proof}
From Eq. (\ref{Equa22}), it is easy to see that the $(1,1)$- and $(2,1)$-blocks in the matrix $\mathcal{K}_{\alpha}^{-1}\mathcal{A}$ tend to $I$ and zero,  respectively.  This means that $n$ eigenvalues of the matrix $\mathcal{M}_{\alpha}^{-1}\mathcal{A}$ approaches to $1$ as $\alpha\rightarrow 0$. The rest of the eigenvalues are equal to those of the matrix $K_{\alpha}$. Let $(\gamma_{\alpha} , u_{\alpha})$ be an eigenpair of the matrix $K_{\alpha}$, where $K_{\alpha}=S_{\alpha}^{-1}B(H_{\alpha}+A)^{-1}B^T$ and
$S_{\alpha}=Q+B(H_{\alpha}+A)^{-1}B^T$. Therefore, we have $K_{\alpha}u_{\alpha}=\gamma_{\alpha} u_{\alpha}$, which is equivalent to
\begin{equation*}
B(H_{\alpha}+A)^{-1}B^Tu_{\alpha}=\gamma_{\alpha} \left(Q+B(H_{\alpha}+A)^{-1}B^T\right)u_{\alpha}.
\end{equation*}
Multiplying both sides of the latter equation by $u_{\alpha}^*$ gives
\begin{eqnarray}
\nonumber\gamma_{\alpha} &\h=\h& \frac{ u_{\alpha}^*B(H_{\alpha}+A)^{-1}B^Tu_{\alpha} }{u^*_{\alpha} \left(Q+B(H_{\alpha}+A)^{-1}B^T\right)u_{\alpha}}\\
\nonumber &\h=\h& \dfrac{u_{\alpha}^*B(H_{\alpha}+A)^{-1}B^Tu_{\alpha}}{u_{\alpha}^*Qu_{\alpha}+u_{\alpha}^*B(H_{\alpha}+A)^{-1}B^Tu_{\alpha}}\\
&\h=\h& \dfrac{r_{\alpha}}{q_{\alpha}+r_{\alpha}},\label{EqREig}
\end{eqnarray}
where $r_{\alpha}=u_{\alpha}^*B(H_{\alpha}+A)^{-1}B^Tu_{\alpha}$ and $q_{\alpha}=u_{\alpha}^*Qu_{\alpha}$. Since $H_{\alpha}$ is symmetric positive definite and $A$ is nonsymmetric positive definite,
we deduce that $\Re(r_{\alpha})\geqslant 0$. On the other hand, since $Q$ is symmetric positive definite, we conclude that $q_{\alpha}>0$.
\end{proof}
\end{thm}

In the section of the numerical experiments,  we choose the matrix $Q$ as $Q=\alpha I+\beta BB^T$ where $\beta>0$. In this case, when $\alpha,\beta\rightarrow 0$ we deduce that $n$ eigenvalues of the preconditioned matrix are approximately equal to 1 and from \eqref{EqREig} it follows that the others tends to $0$ or $1$.

In the sequel, we present the eigenvalue analysis of the preconditioned matrix when the matrix $A$ is symmetric positive definite.
To do so, analogous to Theorems 3.2 and 3.3  in \cite{Ren}, we give the next theorems to discuss the eigenvalue distribution of $\mathcal{K}^{-1}\mathcal{A}$.

\begin{thm}\label{Theo4}
Suppose that  $A,H\in \mathbb{R}^{n\times n}$ and $Q\in \mathbb{R}^{m\times m}$ are symmetric positive definite matrices and $B\in \mathbb{R}^{m\times n}$ is rank-deficient. Then all the nonzero eigenvalues having nonzero imaginary parts of the preconditioned matrix $\mathcal{K}^{-1}\mathcal{A}$ are located in a circle centered at $(1,0)$ with radius $\sqrt{\frac{\lambda_{\max}(H)}{\lambda_{\max}(H)+\lambda_{\min}(A)}}$.
\begin{proof}
The proof is similar to that of Theorem 3.2 in \cite{Ren}. Let
\begin{equation*}
\mathcal{T}_0=\left(\begin{matrix}
H+A & 0\\
0 & Q
\end{matrix}\right), \quad \mathcal{I}_0=\left(\begin{matrix}
I & 0\\
0 & -I
\end{matrix}\right).
\end{equation*}
Obviously, the matrix $\mathcal{T}_0$ is symmetric positive definite. Hence, the matrix $\mathcal{K}^{-1}\mathcal{A}$ is similar to
\begin{equation}\label{eq3.2}
\mathcal{T}^{\frac{1}{2}}_0 \mathcal{K}^{-1}\mathcal{A} \mathcal{T}^{-\frac{1}{2}}_0=\left(\mathcal{T}^{-\frac{1}{2}}_0\mathcal{I}_0\mathcal{K}\mathcal{T}^{-\frac{1}{2}}_0\right)^{-1}\left(\mathcal{T}^{-\frac{1}{2}}_0\mathcal{I}_0\mathcal{A}\mathcal{T}^{-\frac{1}{2}}_0\right)=\left(\begin{matrix}
I & \bar{B}^T\\
\bar{B} & -I
\end{matrix}\right)^{-1}\left(\begin{matrix}
\bar{A} & \bar{B}^T\\
\bar{B} & 0
\end{matrix}\right),
\end{equation}
where $\bar{A}=(H+A)^{-\frac{1}{2}}A(H+A)^{-\frac{1}{2}}$ and $\bar{B}=Q^{-\frac{1}{2}}B(H+A)^{-\frac{1}{2}}$. Eq. \eqref{eq3.2} can be rewritten as
\[
\mathcal{T}^{\frac{1}{2}}_0 \mathcal{K}^{-1}\mathcal{A} \mathcal{T}^{-\frac{1}{2}}_0=\left(\begin{matrix}
I & 0\\
-\bar{B} & I
\end{matrix}\right)^{-1}\left(\begin{matrix}
I+\bar{B}^T\bar{B} & 0\\
0 & -I
\end{matrix}\right)^{-1}\left(\begin{matrix}
I & -\bar{B}^T\\
0 & I
\end{matrix}\right)^{-1}\left(\begin{matrix}
\bar{A} & \bar{B}^T\\
\bar{B} & 0
\end{matrix}\right),
\]
which is similar to
\begin{align*}
& \left[\left(\begin{matrix}
I &-\bar{B}^T\\
 0 & I
\end{matrix}\right)\left(\begin{matrix}
I+\bar{B}^T\bar{B} & 0\\
0 & -I
\end{matrix}\right)\right]^{-1}\left(\begin{matrix}
\bar{A} & \bar{B}^T\\
\bar{B} & 0
\end{matrix}\right)\left(\begin{matrix}
I & 0\\
-\bar{B} & I
\end{matrix}\right)^{-1}\\
&=\left(\begin{matrix}
I+\bar{B}^T\bar{B} & \bar{B}^T\\
0 & -I
\end{matrix}\right)^{-1}\left(\begin{matrix}
\bar{A}+\bar{B}^T\bar{B} & \bar{B}^T\\
\bar{B} & 0
\end{matrix}\right)\\
&\triangleq \hat{\mathcal{K}}^{-1}\hat{\mathcal{A}}.
\end{align*}
Evidently, the matrix
\[
\mathcal{T}_1=\left(\begin{matrix}
I+\bar{B}^T\bar{B} & 0\\
0 & I
\end{matrix}\right).
\]
is symmetric positive definite. Therefore, the matrix $\hat{\mathcal{K}}^{-1}\hat{\mathcal{A}}$ is similar to
\begin{equation*}
\mathcal{T}_1^{\frac{1}{2}}\hat{\mathcal{K}}^{-1}\hat{\mathcal{A}}\mathcal{T}_1^{-\frac{1}{2}}= \left(\mathcal{T}_1^{-\frac{1}{2}}\hat{\mathcal{K}}\mathcal{T}_1^{-\frac{1}{2}}\right)^{-1}\left(\mathcal{T}_1^{-\frac{1}{2}}\hat{\mathcal{A}}\mathcal{T}_1^{-\frac{1}{2}}\right)= \left(\begin{matrix}
I & \tilde{B}^T\\
0 & -I
\end{matrix}\right)^{-1}\left(\begin{matrix}
\tilde{A} & \tilde{B}^T\\
\tilde{B} & 0
\end{matrix}\right)\triangleq \tilde{\mathcal{K}}^{-1}\tilde{\mathcal{A}},
\end{equation*}
where $\tilde{A}= (I+\bar{B}^T\bar{B})^{-\frac{1}{2}}(\bar{A}+\bar{B}^T\bar{B})(I+\bar{B}^T\bar{B})^{-\frac{1}{2}}$, $\tilde{B}=\bar{B}(I+\bar{B}^T\bar{B})^{-\frac{1}{2}}$. Hence, we deduce that ${\mathcal{K}}^{-1}\bar{\mathcal{A}}$ is similar to $\tilde{\mathcal{K}}^{-1}\tilde{\mathcal{A}}$. Hence, it is enough  to analyse the eigenvalue distribution of $\tilde{\mathcal{K}}^{-1}\tilde{\mathcal{A}}$.

First of all, let $(\lambda, (u;v))$ be an eigenpair of the matrix $\tilde{\mathcal{K}}^{-1}\tilde{\mathcal{A}}$. Then, we have
\begin{equation*}
\left(\begin{matrix}
\tilde{A} & \tilde{B}^T\\
\tilde{B} & 0
\end{matrix}\right)\left(\begin{matrix}
u\\
v
\end{matrix}\right)=\lambda \left(\begin{matrix}
I & \tilde{B}^T\\
0 & -I
\end{matrix}\right)\left(\begin{matrix}
u\\
v
\end{matrix}\right),
\end{equation*}
which can be rewritten as
\begin{equation}\label{eq3.4}
\begin{cases}
\tilde{A}u+\tilde{B}^Tv=\lambda u+ \lambda \tilde{B}^Tv,\\
\tilde{B}u=-\lambda v.
\end{cases}
\end{equation}
If $v=0$, then  from the first equality of \eqref{eq3.4}, we obtain $\tilde{A}u=\lambda u$.
This shows that $\lambda$ is real, because  $\tilde{A}$ is symmetric positive definite. If $u=0$, then from the second equality of \eqref{eq3.4} we get
$-\lambda v=0$. Therefore, we have $\lambda=0$, since $v$ cannot be zero.

Now, we assume that $u\neq 0$ and $v\neq 0$ with $\norm{u}^2_2+\norm{v}^2_2=1$. Multiplying both sides of the first equality of \eqref{eq3.4} by $u^*$ gives
\begin{equation}\label{eq3.6}
u^*\tilde{A}u-\lambda\norm{u}^2_2=(\lambda-1)u^*\tilde{B}^Tv.
\end{equation}
Multiplying the second equation of Eq. \eqref{eq3.4} by $v^*$ results in $u^*\tilde{B}^Tv=-\bar{\lambda}\|v\|_2^2$. Substituting this into Eq. \eqref{eq3.6} yields
\begin{equation}\label{eq3.8}
u^*\tilde{A}u +|\lambda |^2 \norm{v}^2_2-\lambda+(\lambda-\bar{\lambda})\norm{v}^2_2=0.
\end{equation}
Defining $\lambda=a+ib$, the imaginary part of Eq. \eqref{eq3.8} is written as
\[
b(2\norm{v}^2_2-1)=0.
\]
From this equation we deduce that $b=0$ or $\norm{v}^2_2=\frac{1}{2}$. If $b=0$, then $\lambda$ is real. If $b\neq 0$, then we get $\norm{v}^2_2=\norm{u}^2_2=\frac{1}{2}$. This, together with Eq. \eqref{eq3.8} gives
\[
2u^*\tilde{A}u+|\lambda|^2-\lambda-\bar{\lambda}=0.
\]
By some computations and using the Courant-Fisher min-max theorem \cite{Saadbook}, we can write
\begin{equation}\label{eq3.10}
|\lambda -1|^2=1-2u^*\tilde{A}u=1-\dfrac{u^*\tilde{A}u}{u^*u}\leqslant 1-\lambda_{\min}(\tilde{A}).
\end{equation}
It is not difficult to verify that the matrix $\tilde{A}$ is similar to $(H+A+B^TQ^{-1}B)^{-1}(A+B^TQ^{-1}B)$. Suppose that $(\tilde{\lambda},\tilde{x})$ is an eigenpair of this matrix. Therefore, we have
\begin{equation}\label{eqabove}
(A+B^TQ^{-1}B)\tilde{x}=\tilde{\lambda}(H+A+B^TQ^{-1}B)\tilde{x}.
\end{equation}
Multiplying both sides of Eq. \eqref{eqabove} by $\tilde{x}^*$ and some simplifying we obtain
\begin{align}
\nonumber\tilde{\lambda}&=\dfrac{x^*Ax+x^*B^TQ^{-1}Bx}{x^*Hx+x^*Ax+x^*B^TQ^{-1}Bx}\\
\nonumber&\geqslant \dfrac{x^*Ax}{x^*Hx+x^*Ax}=\dfrac{\frac{x^*Ax}{x^*x}}{\frac{x^*Hx}{x^*x}+\frac{x^*Ax}{x^*x}}\\
&\geqslant \dfrac{\lambda_{\min}(A)}{\lambda_{\max}(H)+\lambda_{\min}(A)}.\label{eq3.10.1}
\end{align}
It follows from Eqs. \eqref{eq3.10} and \eqref{eq3.10.1} that
\[
|\lambda -1|^2 \leqslant 1-\dfrac{\lambda_{\min}(A)}{\lambda_{\max}(H)+\lambda_{\min}(A)}=\dfrac{\lambda_{\max}(H)}{\lambda_{\max}(H)+\lambda_{\min}(A)}<1,
\]
which shows that all nonzero eigenvalues having nonzero imaginary parts of the preconditioned matrix ${\mathcal{K}}^{-1}\mathcal{A}$ are located in a circle centered at (1,0) with radius $\sqrt{\frac{\lambda_{\max}(H)}{\lambda_{\max}(H)+\lambda_{\min}(A)}}$ which is strictly less than one.
\end{proof}
\end{thm}

\begin{cor}\label{Cor1}
Let $A,H\in \mathbb{R}^{n\times n}$ and $Q\in \mathbb{R}^{m\times m}$ be symmetric positive definite matrices and $B\in \mathbb{R}^{m\times n}$ be rank-deficient. Then, all the nonzero eigenvalues having nonzero imaginary parts of the preconditioned matrix ${\mathcal{K}}^{-1}\mathcal{A}$ are located in the following domain
\[
\mathcal{D}=\left\{\lambda\in \mathbb{C}: |\lambda-\frac{1}{2}|<\frac{1}{2}\right\} \cap \left\{\lambda\in \mathbb{C}:  |\lambda-1|\leqslant\sqrt{\frac{\lambda_{\max}(H)}{\lambda_{\max}(H)+\lambda_{min}(A)}}\right\}.
\]
\end{cor}


\begin{thm}\label{Theo5}
Let $A,H\in \mathbb{R}^{n\times n}$ and $Q\in \mathbb{R}^{m\times m}$ be symmetric positive definite matrices and $B\in \mathbb{R}^{m\times n}$ be rank-deficient. Let also $\sigma_{\min}$ and  $\sigma_{\max}$ be the smallest and the largest nonzero singular values of the matrix $B$, respectively. Then, all the  nonzero real eigenvalues of the matrix $\mathcal{K}^{-1}\mathcal{A}$ are located in
\begin{align*}
\left[ \min \left\{\dfrac{\lambda_{\min}(A)}{\lambda_{\max}(H)+\lambda_{\min}(A)}, \dfrac{\sigma^2_{\min}}{\lambda_{\max}(Q)(\lambda_{\max}(H)+\kappa(H)\lambda_{\max}(A))+\sigma^2_{\min}} \right\}\right.,\\ \left. \dfrac{\lambda_{\min}(Q)\lambda_{\max}(A)+\sigma_{\max}^2}{\lambda_{\min}(Q)(\lambda_{\min}(H)+\lambda_{\max}(A))+\sigma_{\max}^2} \right].
\end{align*}
\end{thm}
\begin{proof}
Since ${\mathcal{K}}^{-1}\mathcal{A}$ is similar to $\tilde{\mathcal{K}}^{-1}\tilde{\mathcal{A}}$, we only need to study the nonzero real  eigenvalues of the matrix $\tilde{\mathcal{K}}^{-1}\tilde{\mathcal{A}}$ which are the same as those of the matrix $\tilde{\mathcal{A}}\tilde{\mathcal{K}}^{-1}$. Since $\tilde{A}$ is symmetric positive definite and it is similar to $(H+A+B^TQ^{-1}B)^{-1}(A+B^TQ^{-1}B)$, we deduce that all the eigenvalues of $\tilde{A}$ are positive and less than one. On the other hand,
\begin{align*}
\tilde{\mathcal{A}}\tilde{\mathcal{K}}^{-1} &=\left(\begin{matrix}
\tilde{A}&\tilde{B}^T\\
\tilde{B}&0
\end{matrix}\right)\left(\begin{matrix}
I&\tilde{B}^T\\
0&-I
\end{matrix}\right)^{-1}\\
&=\left(\begin{matrix}
\tilde{A}(\tilde{A}-I)&(\tilde{A}-I)\tilde{B}^T\\
\tilde{B}(\tilde{A}-I)&\tilde{B}\tilde{B}^T
\end{matrix}\right)\left(\begin{matrix}
\tilde{A}-I & 0\\
0 & I
\end{matrix}\right)^{-1}\\
&\triangleq \breve{\mathcal{A}}\breve{\mathcal{S}}^{-1},
\end{align*}
where $\mathcal{\breve{A}}$ and $\breve{\mathcal{S}}$ are symmetric and $\breve{\mathcal{S}}$ is nonsingular. Then, then the eigenvalues of $\tilde{\mathcal{K}}^{-1}\tilde{\mathcal{A}}$ and $\breve{\mathcal{S}}^{-1}\breve{\mathcal{A}}$ are the same.
Assume that $\tilde{A}= X\Lambda X^T$ with $I-\Lambda\succ 0$, where $X$ is an orthogonal matrix. Define
\begin{equation}
\mathcal{Z}=\left(\begin{matrix}
X & 0\\
0 & I
\end{matrix}\right)\quad \textrm{and} \quad  \mathcal{D}=\left(\begin{matrix}
I-\Lambda & 0\\
0 & I
\end{matrix}\right)\succ 0.
\end{equation}
Since $\mathcal{Z}$ is orthogonal and $\mathcal{D}$ is symmetric positive definite, the matrix $\breve{\mathcal{S}}^{-1}\breve{\mathcal{A}}$ is similar to $\mathcal{D}^{\frac{1}{2}}\mathcal{Z}^T\breve{\mathcal{S}}^{-1}\breve{\mathcal{A}}\mathcal{Z}\mathcal{D}^{-\frac{1}{2}}$. We have
\begin{equation}\label{eq3.12}
\mathcal{D}^{\frac{1}{2}}\mathcal{Z}^T\breve{\mathcal{S}}^{-1}\breve{\mathcal{A}}\mathcal{Z}\mathcal{D}^{-\frac{1}{2}}=\left(\mathcal{D}^{-\frac{1}{2}}\mathcal{Z}^T\breve{\mathcal{S}}\mathcal{Z}\mathcal{D}^{-\frac{1}{2}}\right)^{-1}\left(\mathcal{D}^{-\frac{1}{2}}\mathcal{Z}^T\breve{\mathcal{A}}\mathcal{Z}\mathcal{D}^{-\frac{1}{2}}\right).
\end{equation}
Let $(\lambda,w)$ be an eigenpair of the matrix $\mathcal{D}^{\frac{1}{2}}\mathcal{Z}^T\breve{\mathcal{S}}^{-1}\breve{\mathcal{A}}\mathcal{Z}\mathcal{D}^{-\frac{1}{2}}$ such that $\lambda\neq 0$. Thus, from Eq. \eqref{eq3.12} it holds that
\[
\mathcal{D}^{-\frac{1}{2}}\mathcal{Z}^T\breve{\mathcal{A}}\mathcal{Z}\mathcal{D}^{-\frac{1}{2}}w=\lambda \mathcal{D}^{-\frac{1}{2}}\mathcal{Z}^T\breve{\mathcal{S}}\mathcal{Z}\mathcal{D}^{-\frac{1}{2}}w,
\]
which is equivalent to
\begin{equation}\label{eq3.13}
\left(\begin{matrix}
-\Lambda & Q^T\\
Q & P
\end{matrix}\right)w=\lambda\left(\begin{matrix}
-I & 0\\
0 & I
\end{matrix}\right)w,
\end{equation}
where $Q=-\tilde{B}X(I-\Lambda)^{-\frac{1}{2}}$ and $P=\tilde{B}\tilde{B}^T$. Without loss of generality, assume that $w=(u;v)$ such that $\norm{u}^2_2+\norm{v}^2_2=1$. Therefore, we can rewrite Eq. \eqref{eq3.13} as
\begin{equation}\label{eq3.14}
\begin{cases}
-\Lambda u +Q^Tv=-\lambda u,\\
Qu+Pv=\lambda v.
\end{cases}
\end{equation}
Multiplying both sides of the first and second equality in Eq. \eqref{eq3.14} by $u^*$ and $v^*$, respectively, leads to
\begin{equation}\label{eq3.14.1}
u^*\Lambda u - u^*Q^Tv=\lambda \norm{u}^2_2 \quad\textrm{and}\quad v^*Qu=\lambda \norm{v}^2_2-v^*Pv.
\end{equation}
Combining the two equations in \eqref{eq3.14.1} and using $\norm{u}^2_2+\norm{v}^2_2=1$, eventuate
\begin{equation}\label{eq3.15}
u^*\Lambda u+v^*Pv-\bar{\lambda}+(\bar{\lambda}-\lambda)\norm{v}^2_2=0.
\end{equation}
By considering the real part of Eq. \eqref{eq3.15} and the , we see that
\begin{equation}\label{eq3.16}
a=u^*\Lambda u+v^*Pv\leqslant \norm{u}^2_2\lambda_{\max}(\Lambda)+\norm{v}^2_2\lambda_{\max}(P) \leqslant \max\{\lambda_{\max}(\Lambda),\lambda_{\max}(P)\}.
\end{equation}
In the same way, we deduce that
\begin{equation}\label{eq3.16-0}
a\geqslant \min\{\lambda_{\min}(\Lambda),\lambda_{\min}(P)\}.
\end{equation}
On the other hand, the eigenvalues of $\Lambda$ and $\tilde{A}$ are the same. Using the proof of Theorem \ref{Theo4} we want to study the upper bound of the eigenvalues of $\Lambda$. To do this, since $\Lambda$ is symmetric, using the Courant-Fisher min-max theorem  we conclude that
\begin{align}\label{eq3.16.1}
\nonumber\lambda(\Lambda)&=\dfrac{x^*Ax+x^*B^TQ^{-1}Bx}{x^*Hx+x^*Ax+x^*B^TQ^{-1}Bx}\\
\nonumber&\leqslant \dfrac{\lambda_{\max}(A)x^*x+\lambda_{\max}(Q^{-1})x^*B^TBx}{\lambda_{\min}(H)x^*x+\lambda_{\max}(A)x^*x+\lambda_{\max}(Q^{-1})x^*B^TBx}\\
\nonumber&\leqslant \dfrac{\lambda_{\max}(A)x^*x+\lambda_{\max}(Q^{-1})\lambda_{\max}(B^TB)x^*x}{\lambda_{\min}(H)x^*x+\lambda_{\max}(A)x^*x+\lambda_{\max}(Q^{-1})\lambda_{\max}(B^TB)x^*x}\\
&=\dfrac{\lambda_{\min}(Q)\lambda_{\max}(A)+\sigma_{\max}^2}{\lambda_{\min}(Q)(\lambda_{\min}(H)+\lambda_{\max}(A))+\sigma_{\max}^2}.
\end{align}
It follows from the proof of Theorem \ref{Theo4} and Eq. \eqref{eq3.16.1} that
\begin{equation}\label{eq3.17}
\dfrac{\lambda_{\min}(A)}{\lambda_{\max}(H)+\lambda_{\min}(A)}I\preceq \Lambda \preceq \dfrac{\lambda_{\min}(Q)\lambda_{\max}(A)+\sigma_{\max}^2}{\lambda_{\min}(Q)(\lambda_{\min}(H)+\lambda_{\max}(A))+\sigma_{\max}^2}I.
\end{equation}
We can rewrite the matrix $P$ as $P=\tilde{B}\tilde{B}^T=\bar{B}(I+\bar{B}^T\bar{B})^{-1}\bar{B}^T$.
 Let $\bar{B}=U[\Sigma, 0]V^T$ be the singular value decomposition of the matrix $\bar{B}$ such that $U\in \mathbb{R}^{m\times m}$ and $V\in \mathbb{R}^{n\times n}$ are orthogonal matrices and $\Sigma =\textrm{diag}(\tau_1, \tau_2, \cdots , \tau_r,0,\ldots,0)\in \mathbb{R}^{m\times m}$ is a diagonal matrix, where $\tau_1\geqslant \tau_2\geqslant \cdots\geqslant \tau_r>0$ are the nonzero singular values of the matrix $\bar{B}$. Hence,
\begin{align*}
P &=U\Sigma(I+\Sigma^2)^{-1}\Sigma U^T\\
&= U \text{diag}(\frac{\tau^2_1}{1+\tau^2_1}, \ldots, \frac{\tau^2_r}{1+\tau^2_r},0,\ldots, 0)U^T.
\end{align*}
Therefore, the nonzero eigenvalues of $P$ satisfy
\begin{equation}\label{eq3.18}
\frac{\tau^2_r}{1+\tau^2_r}\leqslant \lambda(P)\leqslant \frac{\tau^2_1}{1+\tau^2_1},
\end{equation}
where $\lambda(P)$ is a nonzero eigenvalue of $P$. Obviously, $\tau^2_1$ is the largest eigenvalue of the matrix $\bar{B}\bar{B}^T=Q^{-\frac{1}{2}}B(H+A)^{-1}B^TQ^{-\frac{1}{2}}$. By using Courant-Fisher Min-Max theorem we obtain
\begin{align}\label{eq3.18.1}
\nonumber x^*Q^{-\frac{1}{2}}B(H+A)^{-1}B^TQ^{-\frac{1}{2}}x &=x^*Q^{-\frac{1}{2}}BH^{-\frac{1}{2}}(I+H^{-\frac{1}{2}}AH^{-\frac{1}{2}})^{-1}H^{-\frac{1}{2}}B^TQ^{-\frac{1}{2}}\\
\nonumber &\leqslant \lambda_{\max}(I+S)^{-1}x^*Q^{-\frac{1}{2}}BH^{-1}B^TQ^{-\frac{1}{2}}x\\
\nonumber &\leqslant \frac{1}{1+\lambda_{\min}(S)}~\frac{1}{\lambda_{\min}(H)}~x^*Q^{-\frac{1}{2}}BB^TQ^{-\frac{1}{2}}x \\
\nonumber & \leqslant \frac{1}{1+\lambda_{\min}(S)}~\frac{1}{\lambda_{\min}(H)}~\sigma^2_{\max}~ x^*Q^{-1}x\\
&\leqslant \dfrac{\sigma^2_{\max}}{(1+\lambda_{\min}(S))\lambda_{\min}(H)\lambda_{\min}(Q)}~ x^*x,
\end{align}
where $S=H^{-\frac{1}{2}}AH^{-\frac{1}{2}}$. Furthermore, since $S$ is symmetric, we can write that
\begin{eqnarray*}
          x^*Sx &\h=\h        & x^*H^{-\frac{1}{2}}AH^{-\frac{1}{2}}x \\
                &\h\geqslant\h& \lambda_{\min}(A)x^*H^{-1}x \\
                &\h\geqslant\h& \lambda_{\min}(H^{-1})\lambda_{\min}(A)~x^*x\\
                &\h=\h        &\frac{\lambda_{\min}(A)}{\lambda_{\max}(H)}~x^*x.
\end{eqnarray*}
Thus
\begin{equation}\label{eq3.18.2}
\lambda_{\min}(S)\geqslant \frac{\lambda_{\min}(A)}{\lambda_{\max}(H)}.
\end{equation}
Form Eqs. \eqref{eq3.18.1} and \eqref{eq3.18.2}, it is straightforward to see that
\begin{equation*}
\tau_1^2\leqslant \dfrac{\sigma^2_{\max}}{\lambda_{\min}(Q)(\lambda_{\min}(H)+\frac{1}{\kappa(H)}\lambda_{\min}(A))}.
\end{equation*}
In the same way, we derive that
\begin{equation}\label{eq3.18.3}
\dfrac{\sigma^2_{\min}}{\lambda_{\max}(Q)(\lambda_{\max}(H)+\kappa(H)\lambda_{\max}(A))}\leqslant\tau_r^2\leqslant \tau_1^2\leqslant \dfrac{\sigma^2_{\max}}{\lambda_{\min}(Q)(\lambda_{\min}(H)+\frac{1}{\kappa(H)}\lambda_{\min}(A))}.
\end{equation}
 From Eq. \eqref{eq3.18} and \eqref{eq3.18.3} for the nonzero eigenvalues of $P$ we have
\begin{equation}\label{eq3.19}
\begin{cases}
\lambda_{\min}(P)\geqslant \dfrac{\sigma^2_{\min}}{\lambda_{\max}(Q)(\lambda_{\max}(H)+\kappa(H)\lambda_{\max}(A))+\sigma^2_{\min}},\\
\lambda_{\max}(P)\leqslant\dfrac{\sigma^2_{\max}}{\lambda_{\min}(Q)(\lambda_{\min}(H)+\frac{1}{\kappa(H)}\lambda_{\min}(A))+\sigma^2_{\max}}.
\end{cases}
\end{equation}
Using Eqs. \eqref{eq3.16}, \eqref{eq3.16-0}, \eqref{eq3.17} and \eqref{eq3.19}, for the nonzero real eigenvalues of $\mathcal{K}^{-1}\mathcal{A}$ we evaluate that
\begin{align*}
&\min \left\{\dfrac{\lambda_{\min}(A)}{\lambda_{\max}(H)+\lambda_{\min}(A)}, \dfrac{\sigma^2_{\min}}{\lambda_{\max}(Q)(\lambda_{\max}(H)+\kappa(H)\lambda_{\max}(A))+\sigma^2_{\min}} \right\}\leqslant a \\\nonumber
&\leqslant \max \left\{\dfrac{\lambda_{\min}(Q)\lambda_{\max}(A)+\sigma_{\max}^2}{\lambda_{\min}(Q)(\lambda_{\min}(H)+\lambda_{\max}(A))+\sigma_{\max}^2},
\dfrac{\sigma^2_{\max}}{\lambda_{\min}(Q)(\lambda_{\min}(H)+\frac{1}{\kappa(H)}\lambda_{\min}(A))+\sigma^2_{\max}}\right\}.
\end{align*}
Since we have
\begin{equation*}
\dfrac{\lambda_{\min}(Q)\lambda_{\max}(A)+\sigma_{\max}^2}{\lambda_{\min}(Q)(\lambda_{\min}(H)+\lambda_{\max}(A))+\sigma_{\max}^2}
>\dfrac{\sigma^2_{\max}}{\lambda_{\min}(Q)(\lambda_{\min}(H)+\frac{1}{\kappa(H)}\lambda_{\min}(A))+\sigma^2_{\max}},
\end{equation*}
the desired result is obtained.
\end{proof}

\section{Numerical experiments}\label{Sec5}
In this section, we verify the effectiveness of the GMRES($\ell$) in conjunction with the MGSS preconditioner  for solving singular saddle point problems. All the experiments are carried with some \textsc{Matlab} codes on a machine with 3.20GHz CPU and 8GB RAM.

In all the experiments, the right-hand side vector $b$ is set to $b=\mathcal{A}u^*$, where $u^*\in\mathbb{R}^{m+n}$ is a vector of all ones.
A null vector is always used as an initial guess and the computations are terminated when the 2-norm of the residual vector of the preconditioned system  is reduced by  a factor of $10^{-7}$ or when the number of iterations exceeds 1000. Throughout this section, ``iters''  and ``CPU'' stand for the number of iterations and the CPU time (in seconds) for the semi-convergence, respectively. Moreover, ``$R_k$" is defined as
 \[
 R_k=\frac{\|\mathcal{Q}^{-1}r^{(k)}\|_2}{\|\mathcal{Q}^{-1}r^{(0)}\|_2},
 \]
where $\mathcal{Q}$ is the GSS or MGSS preconditioner,  $r^{(k)}=b-\mathcal{A}u^{(k)}$  ($u^{(k)}$ is the computed solution) and $r^{(0)}=b-\mathcal{A}u^{(0)}$. It is noted
that when the GMRES($\ell$) without preconditioning is used the matrix $\mathcal{Q}$ is set to the identity matrix.  Meanwhile, ``--'' shows that the method fails to converge in at most 1000 iterations.  We choose the matrices $H$ and $Q$ as
\[
H=\alpha(A+A^T)\quad  \textrm{and} \quad Q=\alpha I+\beta BB^T,
\]
where $\alpha$ and $\beta$ are two positive numbers.  Obviously, both of the matrices $H$ and $Q$ are symmetric positive definite. To perform the MGSS preconditioner we need to choose the parameters $\alpha$ and $\beta$ appropriately. We use different values of the parameter $\alpha$ and $\beta$ for each size of matrices for both of the MGSS and the GSS preconditioners. We use GMRES($5$) in conjunction with the GSS and MGSS preconditioners to solve the saddle point problem \eqref{eq1}. To show the effectiveness of both of the preconditioners we also give the numerical results of GMRES($5$) without preconditioning.

We present two examples and for each of them we report two sets of the numerical results. First, the results of the GMRES($5$) are given when all the subsystems are solved exactly by the LU factorization and for the second set of the numerical results, the problem is solved by the same method and the same preconditioners such that the subsystems with the coefficient matrix $H+A$ are solved by the LU factorization and the subsystem with the coefficient matrix $S$  (step 3 of Algorithm 1) is solved by the GMRES($5$) method. To solve the subsystems with the coefficient matrix $S$ with GMRES(5), we use a zero vector as an initial guess and the iterations are stopped as soon as the 2-norm of the residual is reduced by a factor of $10^{-5}$. We now present the examples.


\begin{example}\label{exp1}\rm
We consider the the Oseen equations of the form
\begin{equation*}
\begin{cases}
-\nu \Delta  u +w\cdot \nabla u + \nabla p = f, \\
\nabla \cdot u = 0.
\end{cases}\textrm{in}~\Omega=(-1, 1)\times(-1, 1) \subset \Bbb{R}^2,
\end{equation*}
which are obtained from linearization of the steady-state Navier-Stokes equations by the Picard iteration.
Here the vector field $w$ is the approximation of $u$ from the previous Picard iteration.  {We take the linear system from the 9th Picard iteration.}
The parameter $\nu>0$ represents viscosity, $\Delta$ is the Laplacian operator, $\nabla$ is the gradient. The test problem is a leaky two dimensional lid-driven cavity problem on the unit square domain. We discretize the test problem by the Q2-Q1 mixed finite element method on uniform grids. We use the IFISS software package developed by Elman et al. \cite{IFISS} to generate  linear systems corresponding to $16\times 16$, $32\times 32$, $64\times 64$ and $128\times 128$ grids. We set $\nu=0.01$.

In Table \ref{tab1} the results have been given when all the subsystems are solved exactly and in Table \ref{tab2} the numerical results, when the subsystem with the coefficient matrix $S$ is solved inexactly by the GMRES($5$) method, have been presented.
From Tables \ref{tab1} and \ref{tab2}, we see that the GMRES($5$) method without preconditioning converges very slowly. These tables show that both of the GSS and the MGSS preconditioners are effective. We observe that for large matrices not only the number of iterations of the MGSS preconditioner is less than that of  the GSS method, but also the CPU time of the MGSS preconditioner is less than that of the GSS preconditioner.

In Figure 1, the eigenvalue distribution of the matrices $\mathcal{A}$, $\mathcal{M}^{-1}\mathcal{A}$ and $\mathcal{P}_{GSS}^{-1}\mathcal{A}$ with $\alpha=10^{-4}$ and $\beta=10^{-3}$ have been displayed. We see that the eigenvalues of the matrix $\mathcal{M}^{-1}\mathcal{A}$ are   {more clustered than those of the matrices} $\mathcal{P}_{GSS}^{-1}\mathcal{A}$ and $\mathcal{A}$.
\end{example}

\begin{table}[!ht]
\begin{center}
\small
\caption{Numerical results for Example \ref{exp1} when the subsystems are solved exactly.}\label{tab1}
\vspace{0.2cm}
\begin{tabular}{|c|cc|ccc|ccc|ccc|} \hline
\multicolumn{3}{|c|}{}   &  \multicolumn{3}{c|}{MGSS method}  &  \multicolumn{3}{c|}{GSS method}   &  \multicolumn{3}{c|}{GMRES$(5)$ method}  \\
  \multicolumn{3}{|c|}{}      & \multicolumn{3}{c|}{}  &  \multicolumn{3}{c|}{}  &  \multicolumn{3}{c|}{}  \\ \hline
grid            & $\alpha$ & $\beta$ &Iters & CPU  &  $R_k$    & Iters & CPU & $R_k$   & Iters & CPU & $R_k$ \\ \hline
16$\times$16      & 1e-3 & 1e-2 & 1(3) & 0.05 &  7.30e-9  & 2(2)  & 0.02 &   4.55e-8   & 126(3)& 0.11  &9.92e-8\\
                          & 1e-3 & 1e-3 & 1(3) & 0.05 &  6.65e-9  & 2(1)  & 0.02 &   3.81e-8   &&& \\
                          & 1e-3 & 1e-4 & 1(3) & 0.05 &  6.65e-9  & 2(1)  & 0.02 &   2.57e-8   &&&\\
                          & 1e-2 & 1e-3 & 1(5) & 0.05 &  5.91e-9  & 3(5)  & 0.02 &   5.55e-8   &&&\\
                          & 1e-4 & 1e-3 & 1(2) & 0.05 &  1.72e-8  & 1(4)  & 0.01 &   4.67e-9   &&&\\ [1mm] \hline
32$\times$32      & 1e-3 & 1e-2 & 1(3) & 0.25  &  5.72e-8 & 3(3)  & 0.25  & 4.12e-8 & 385(3) & 0.61 & 9.96e-8 \\
                          & 1e-3 & 1e-3 & 1(3) & 0.25  &  5.62e-8 & 2(5) & 0.23 &  2.57e-8 &  &  &\\
                          & 1e-3 & 1e-4 & 1(3) & 0.25  &  5.60e-8 & 2(4) & 0.23  &  2.55e-8 &  &  &\\
                          & 1e-2 & 1e-3 & 2(1) & 0.26 &  3.21e-8  & 7(4) & 0.31  &  7.64e-8 &&&   \\
                          & 1e-4 & 1e-3 & 1(2) & 0.25 &  4.85e-8  & 1(5) & 0.22  &  3.81e-8 &&&            \\[1mm] \hline
64$\times$64      & 1e-3 & 1e-2 & 1(4)  & 7.60 &  2.84e-8 & 8(3) & 8.62 & 8.09e-8 &725(2) & 5.75  &1.00-7\\
                          & 1e-3 & 1e-3 & 1(4)  & 7.69 &  2.82e-8 & 5(3)  & 8.11 & 9.47e-8 &  &  &\\
                          & 1e-3 & 1e-4 & 1(4)  & 7.53 &  2.82e-8 & 4(1) & 8.01   & 9.12e-8  &  &  &\\
                          & 1e-2 & 1e-3 & 2(4)  & 7.69 &  2.36e-8& 16(5) & 9.85 & 9.68e-8  &&&\\
                          & 1e-4 & 1e-3 & 1(3)  & 7.50 & 7.67e-10 & 2(4) &  7.64  & 8.37e-9  & &&\\[1mm] \hline
 128$\times$128  & 1e-3 & 1e-2 & 1(5)  & 287.85 & 6.86e-8  & 27(5) & 345.54 & 9.94e-8 & -- &--  &--\\
                          & 1e-3 & 1e-3 & 1(5)  & 283.78 & 6.85e-8  & 15(5)  & 318.97   & 9.69e-8  &  &  &\\
                          & 1e-3 & 1e-4 & 1(5)  & 283.63 & 6.85e-8  & 9(3)  &  306.38  & 8.89e-8 &  &  &\\
                          & 1e-2 & 1e-3 & 15(5) & 287.96 & 7.34e-8 & 63(5) & 433.31 & 9.87e-8 & &&\\
                          & 1e-4 & 1e-3 & 1(5)  & 278.85 & 1.31e-8 & 4(2) & 293.27 & 6.90e-8  & &&\\ \hline

\end{tabular}
\end{center}
\end{table}

\begin{table}[!ht]
\begin{center}
\small
\caption{Numerical results for Example \ref{exp1} when the subsystem with the coefficient matrix $S$ is solved inexactly by the GMRES($5$).}\label{tab2}
\vspace{0.2cm}
\begin{tabular}{|c|cc|ccc|ccc|ccc|} \hline
\multicolumn{3}{|c|}{}   &  \multicolumn{3}{c|}{MGSS method}  &  \multicolumn{3}{c|}{GSS method}   &  \multicolumn{3}{c|}{GMRES(5) method}  \\
  \multicolumn{3}{|c|}{}      & \multicolumn{3}{c|}{}  &  \multicolumn{3}{c|}{}  &  \multicolumn{3}{c|}{}  \\ \hline
grid            & $\alpha$ & $\beta$ &Iters & CPU  &  $R_k$    & Iters & CPU & $R_k$   & Iters & CPU & $R_k$ \\ \hline
16$\times$16      & 1e-3 & 1e-2 & 1(3) & 0.23 &  7.38e-9  & 2(2)  & 0.29 &   4.55e-8   & 126(3)& 0.09  &9.92e-8\\
                          & 1e-3 & 1e-3 & 1(3) & 0.23 &  6.73e-9  & 2(1)  & 0.27 &   3.68e-8   &&& \\
                          & 1e-3 & 1e-4 & 1(3) & 0.23 &  6.67e-9  & 2(1)  & 0.27 &   2.62e-8   &&&\\
                          & 1e-2 & 1e-3 & 1(5) & 0.30 &  1.81e-8  & 3(5)  & 0.23 &   5.54e-8   &&&\\
                          & 1e-4 & 1e-3 & 1(2) & 0.18 &  5.75e-11& 1(4)  & 0.19 &   4.66e-9   &&&\\ [1mm] \hline
32$\times$32      & 1e-3 & 1e-2 & 1(3) & 0.44  &  5.75e-8 & 3(3)  & 0.82  &   4.12e-8 & 385(3) & 0.60 & 9.95e-8 \\
                          & 1e-3 & 1e-3 & 1(3) & 0.43  &  5.64e-8 & 2(5) & 0.77  &  2.56e-8 &  &  &\\
                          & 1e-3 & 1e-4 & 1(3) & 0.43  &  5.63e-8 & 2(4) & 0.71  &  2.55e-8 &  &  &\\
                          & 1e-2 & 1e-3 & 2(1) & 0.70  &  3.41e-8  & 7(3) & 2.24  &  9.51e-8 &&&   \\
                          & 1e-4 & 1e-3 & 1(2) & 0.35  &  4.97e-8  & 1(5) & 0.47  &  3.81e-8 &&&            \\[1mm] \hline
64$\times$64      & 1e-3 & 1e-2 & 1(4)  & 5.39 &  2.84e-8 & 8(3) & 15.98   & 8.11e-8 &725(2) & 5.85  &1.00-7\\
                          & 1e-3 & 1e-3 & 1(4)  & 5.19 &  2.82e-8 & 5(3)  & 17.45  & 9.40e-8 &  &  &\\
                          & 1e-3 & 1e-4 & 1(4)  & 5.28 &  2.82e-8 & 4(1) & 13.40  & 9.29e-8  &  &  &\\
                          & 1e-2 & 1e-3 & 2(4)  & 4.65 &  7.77e-10 & 16(5) & 66.22 & 9.70e-8  &&&\\
                          & 1e-4 & 1e-3 & 1(3)  & 5.87 &  2.32e-9 & 2(4) &  7.72  & 8.37e-9  & &&\\[1mm] \hline
 128$\times$128  & 1e-3 & 1e-2 & 3(1)  & 37.83 & 6.86e-8  & 27(5)  & 174.84 & 9.89e-8 & 4287(1) &74.17  &9.99-8\\
                          & 1e-3 & 1e-3 & 3(1)  & 37.62 & 6.85e-8  & 15(5)  & 383.05   & 9.70e-8  &  &  &\\
                          & 1e-3 & 1e-4 & 3(1)  & 37.98 & 6.85e-8  & 9(3) & 486.91 & 8.77e-8 &  &  &\\
                          & 1e-2 & 1e-3 & 3(5) & 34.30 & 7.37-8 & 63(1) & 732.82 & 9.90e-8 & &&\\
                          & 1e-4 & 1e-3 & 1(5)  & 34.56 &  1.07e-8 & 4(2)  & 84.43 & 6.93e-8  & &&\\ \hline

\end{tabular}
\end{center}
\end{table}

\begin{figure}[!ht]
\centerline{\includegraphics[height=5cm,width=5.5cm]{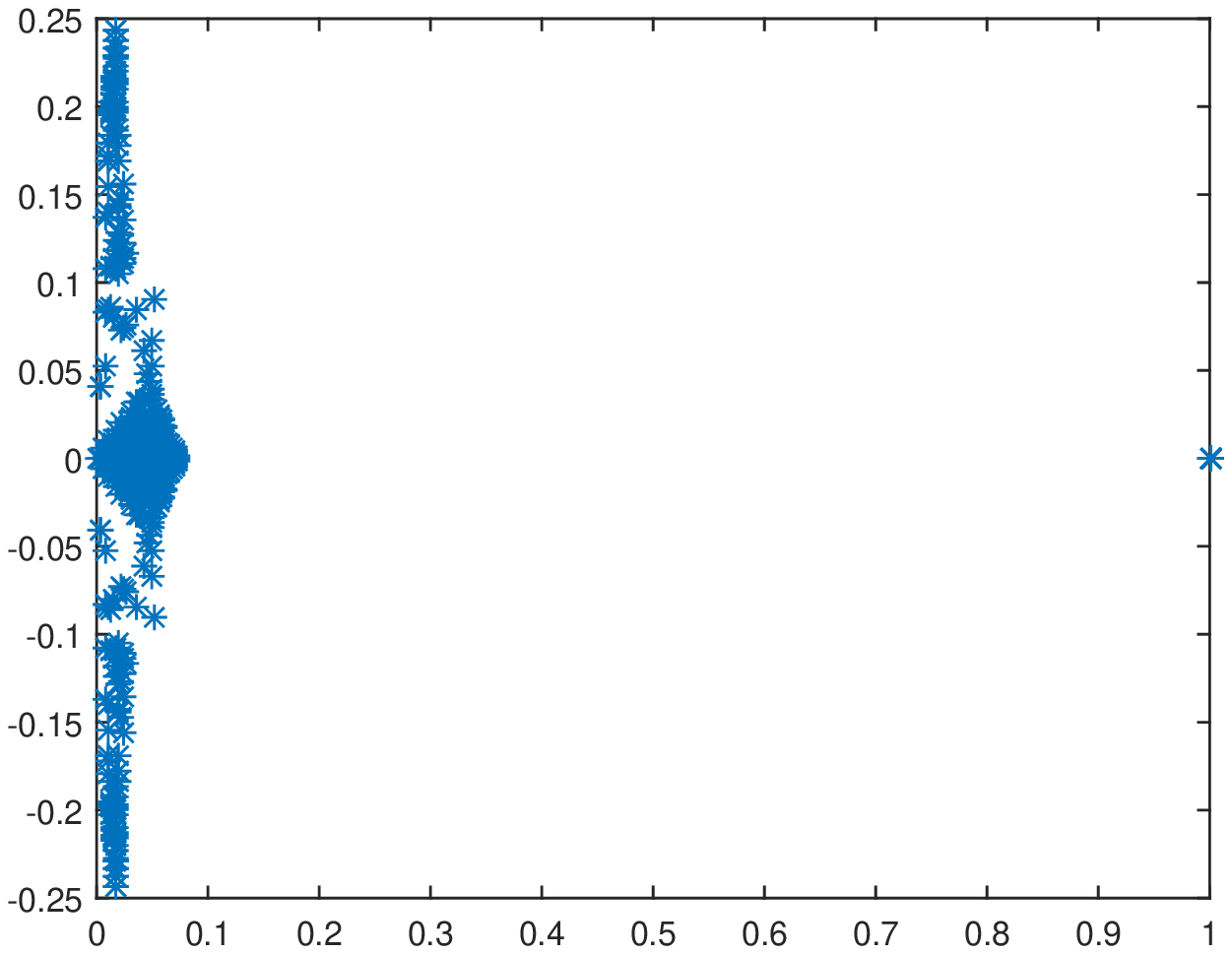}\includegraphics[height=5cm,width=5.5cm]{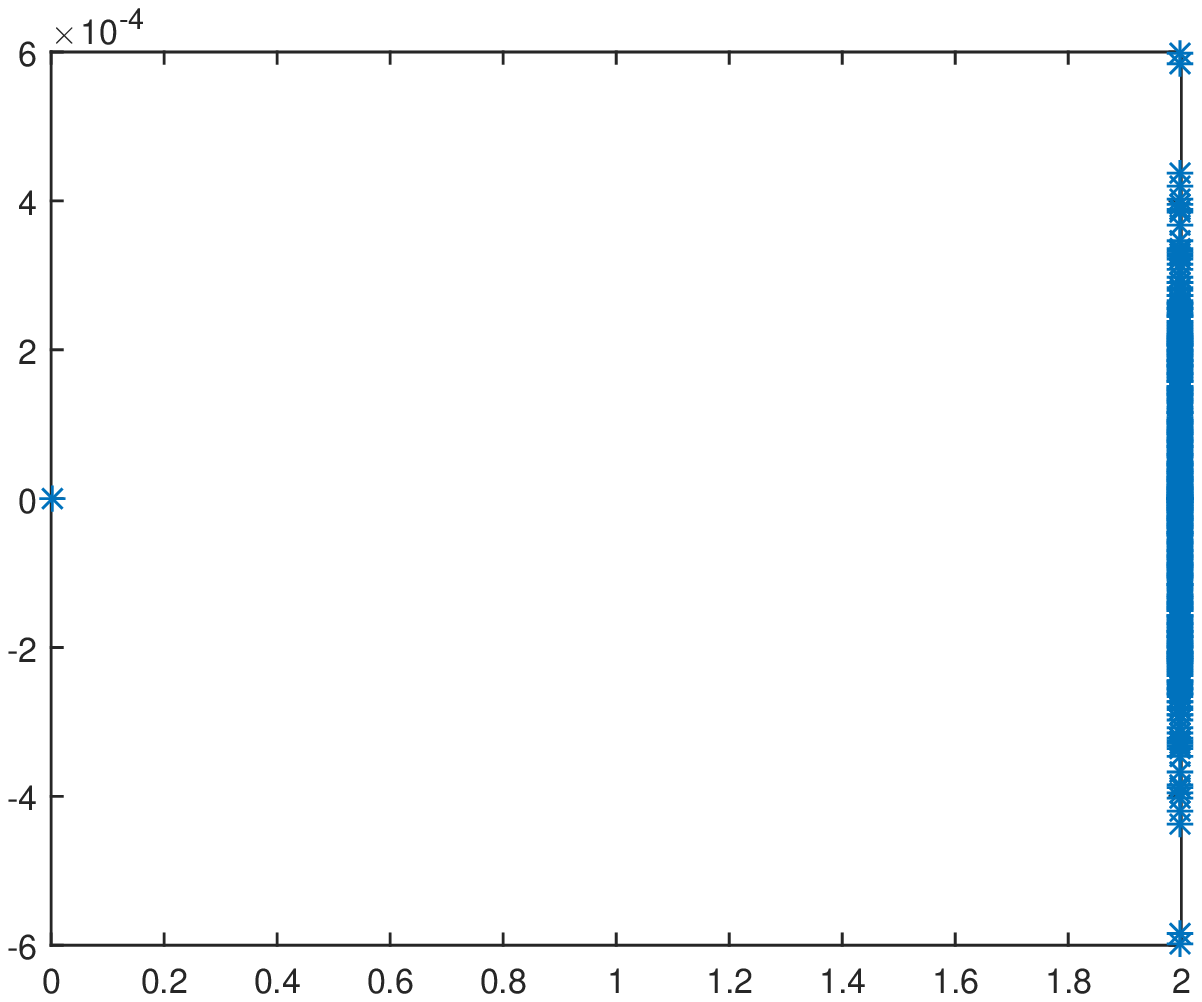}\includegraphics[height=5cm,width=5.5cm]{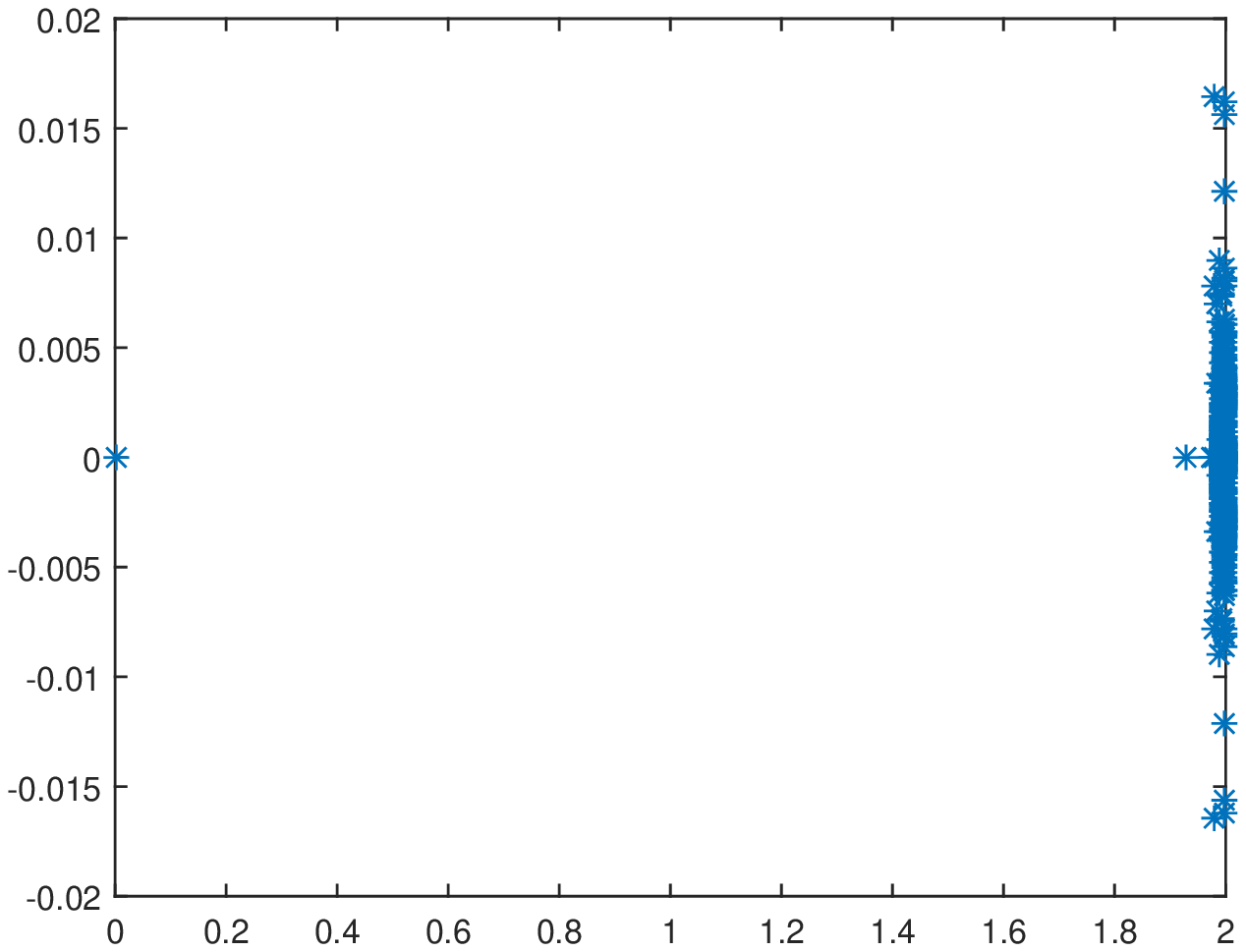}}
{\caption{The eigenvalue distribution of the matrix $\mathcal{A}$ (left), the preconditioned matrices $\mathcal{M}^{-1}\mathcal{A}$ (middle) and  $\mathcal{P}_{GSS}^{-1}\mathcal{A}$ (right) with $\alpha=10^{-4}$ and $\beta=10^{-3}$ and $\textrm{grid}=16\times 16$ for Example \ref{exp1}.}}
\label{Fig1}
\end{figure}

\begin{example}\label{exp2}\rm
In this example, we consider the Navier-Stokes problem
\begin{equation*}
\begin{cases}
-\nu \Delta  u +u\cdot \nabla u + \nabla p = f, \\
\nabla \cdot u = 0,
\end{cases}\textrm{in}~\Omega,
\end{equation*}
where $\Omega=(-1, 1)\times(-1, 1) \subset \Bbb{R}^2$. The scalar $\nu$ is the viscosity, the vector field $u$ represents the velocity, and $p$ denotes the pressure. We set $\nu=0.1$ and use the IFISS software to discretize the leaky lid-driven cavity problem in a square domain with hybrid linearization, and take a finite element subdivision based on uniform grids of square elements, i.e., $16\times 16$, $32\times 32$, $64\times 64$ and $128\times 128$ grids.  {It is noted that, in our hybrid linearization, at first two Picard steps are done to generate a good starting value for the Newton's method and then the Newton iteration is started \cite{Algor866}. For the test example, we use the system of the 4th iteration of the Newton's method.} We use Q2-Q1 pair to discrete the problem. All the assumptions and notations are the same as those used in Example \ref{exp1}. In Table \ref{tab3} the results of the GMRES($5$) method with the GSS and the MGSS preconditioners, when all the subsystems are solved exactly by the LU factorization,  have been given. This table shows that the MGSS preconditioner outperforms (or at least equals) the GSS preconditioner from point of view of number of iterations, especially for large matrices. Based on the CPU time, for large matrices the MGSS preconditioner is  superior to the GSS preconditioner, however this not the case for small problems.

\begin{table}[!ht]
\begin{center}
\small
\caption{Numerical results for Example \ref{exp2} when the subsystems are solved exactly.}\label{tab3}
\begin{tabular}{|c|cc|ccc|ccc|ccc|} \hline
\multicolumn{3}{|c|}{}   &  \multicolumn{3}{c|}{MGSS method}  &  \multicolumn{3}{c|}{GSS method}   &  \multicolumn{3}{c|}{GMRES(5) method}  \\
  \multicolumn{3}{|c|}{}      & \multicolumn{3}{c|}{}  &  \multicolumn{3}{c|}{}  &  \multicolumn{3}{c|}{}  \\ \hline
grid            & $\alpha$ & $\beta$ &Iters & CPU    &  $R_k$    & Iters   & CPU    &   $R_k$   & Iters & CPU & $R_k$ \\ \hline
16$\times$16      & 1e-3 & 1e-2 & 1(4) &   0.05    &  5.54e-9  & 2(3)   &  0.02   &   3.34e-8   & 78(3)& 0.08  &9.65e-8\\
                          & 1e-3 & 1e-3 & 1(4) &   0.05    &  4.41e-9  & 1(5)   &  0.01  &   3.23e-9   &&& \\
                          & 1e-3 & 1e-4 & 1(4) &   0.05    &  4.31e-9  & 1(4)   &  0.01  &   3.90e-9   &&&\\
                          & 1e-2 & 1e-3 & 2(2) &   0.05    &  4.55e-8  & 2(2)   &  0.01  &   3.20e-8   &&&\\
                          & 1e-4 & 1e-3 & 1(3) &   0.05    & 5.93e-10 & 1(4)   &  0.01   &   1.65e-8   &&&\\ [1mm] \hline
32$\times$32      & 1e-3 & 1e-2 & 1(5) &   0.25    &  8.39e-9  & 3(3)   &  0.24   &   3.86e-8 & 366(5) & 0.57  & 9.96e-8 \\
                          & 1e-3 & 1e-3 & 1(5) &   0.25    &  7.80e-9  & 2(2)   &  0.23   &  1.56e-8 &  &  &\\
                          & 1e-3 & 1e-4 & 1(5) &   0.25    &  7.75e-9  & 1(5)   &  0.22   &  4.43e-8 &  &  &\\
                          & 1e-2 & 1e-3 & 3(1) &   0.27    &  5.81e-8  & 3(3)   &  0.25   &  4.01e-8 &&&   \\
                          & 1e-4 & 1e-3 & 1(3) &   0.25    &  9.27e-9  & 1(5)   &  0.22   &  8.33e-8 &&&            \\[1mm] \hline
64$\times$64      & 1e-3 & 1e-2 & 2(2)  &   7.63    &  2.36e-8  & 7(4)   & 8.44    & 9.00e-8 &872(5) & 6.85  &9.97e-8\\
                          & 1e-3 & 1e-3 & 2(2)  &   7.60    &  2.32e-8  & 3(2)   & 7.91    & 5.82e-8 &  &  &\\
                          & 1e-3 & 1e-4 & 2(2)  &   7.61    &  2.31e-8  & 2(3)   & 7.76     &  5.90e-8  &  &  &\\
                          & 1e-2 & 1e-3 & 4(4)  &   7.98    & 7.34e-8   & 7(4)   & 8.46     &  4.95e-8  &&&\\
                          & 1e-4 & 1e-3 & 1(4)  &   7.54    & 2.16e-9   & 2(4)   & 7.73     &  3.09e-8  & &&\\[1mm] \hline
 128$\times$128  & 1e-3 & 1e-2 & 2(5)  & 280.54  & 3.82e-8   & 24(4) & 335.74 & 9.34e-8 & -- &--  &--\\
                          & 1e-3 & 1e-3 & 2(5)  & 282.18  & 3.79e-8   & 8(2)   & 300.36 & 8.02e-8  &  &  &\\
                          & 1e-3 & 1e-4 & 2(5)  & 283.17  & 3.78e-8   & 3(5)   & 293.27 & 8.73e-8 &  &  &\\
                          & 1e-2 & 1e-3 & 15(5) & 289.48 & 8.96e-8   & 25(4)  & 342.51 & 9.34e-8 & &&\\
                          & 1e-4 & 1e-3 & 1(5)   & 278.99 & 9.94e-8   & 4(2)   & 288.75 & 7.15e-8  & &&\\ \hline

\end{tabular}
\end{center}
\end{table}

In Figure 2, the eigenvalue distribution of the matrices $\mathcal{A}$, $\mathcal{M}^{-1}\mathcal{A}$ and $\mathcal{P}_{GSS}^{-1}\mathcal{A}$ with $\alpha=10^{-4}$ and $\beta=10^{-3}$ have been displayed. This figure shows that the eigenvalues of the preconditioned matrix $\mathcal{M}^{-1}\mathcal{A}$ are more clustered than those of the other matrices.

\begin{figure}[!ht]
\centerline{\includegraphics[height=5cm,width=5.5cm]{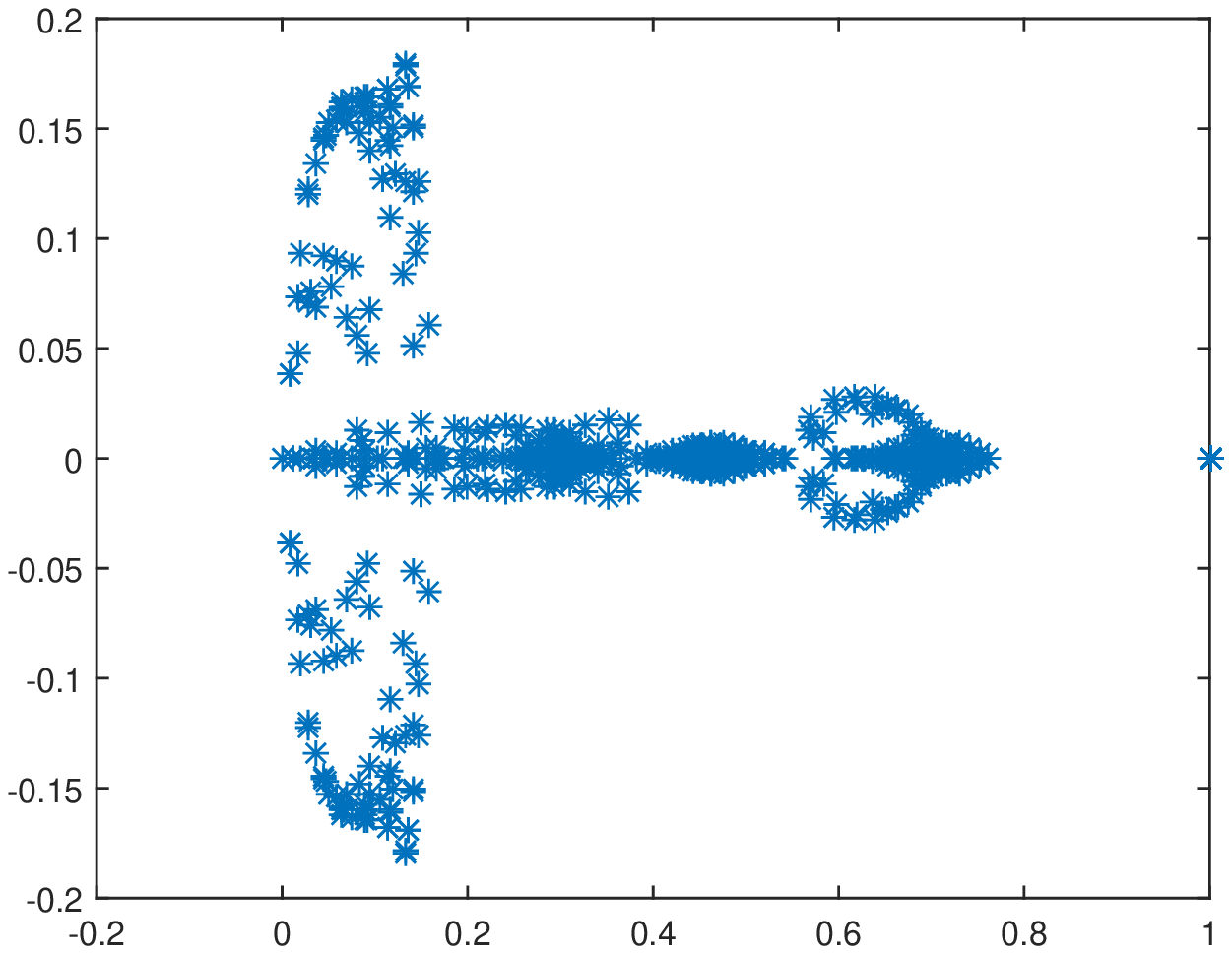}\includegraphics[height=5cm,width=5.5cm]{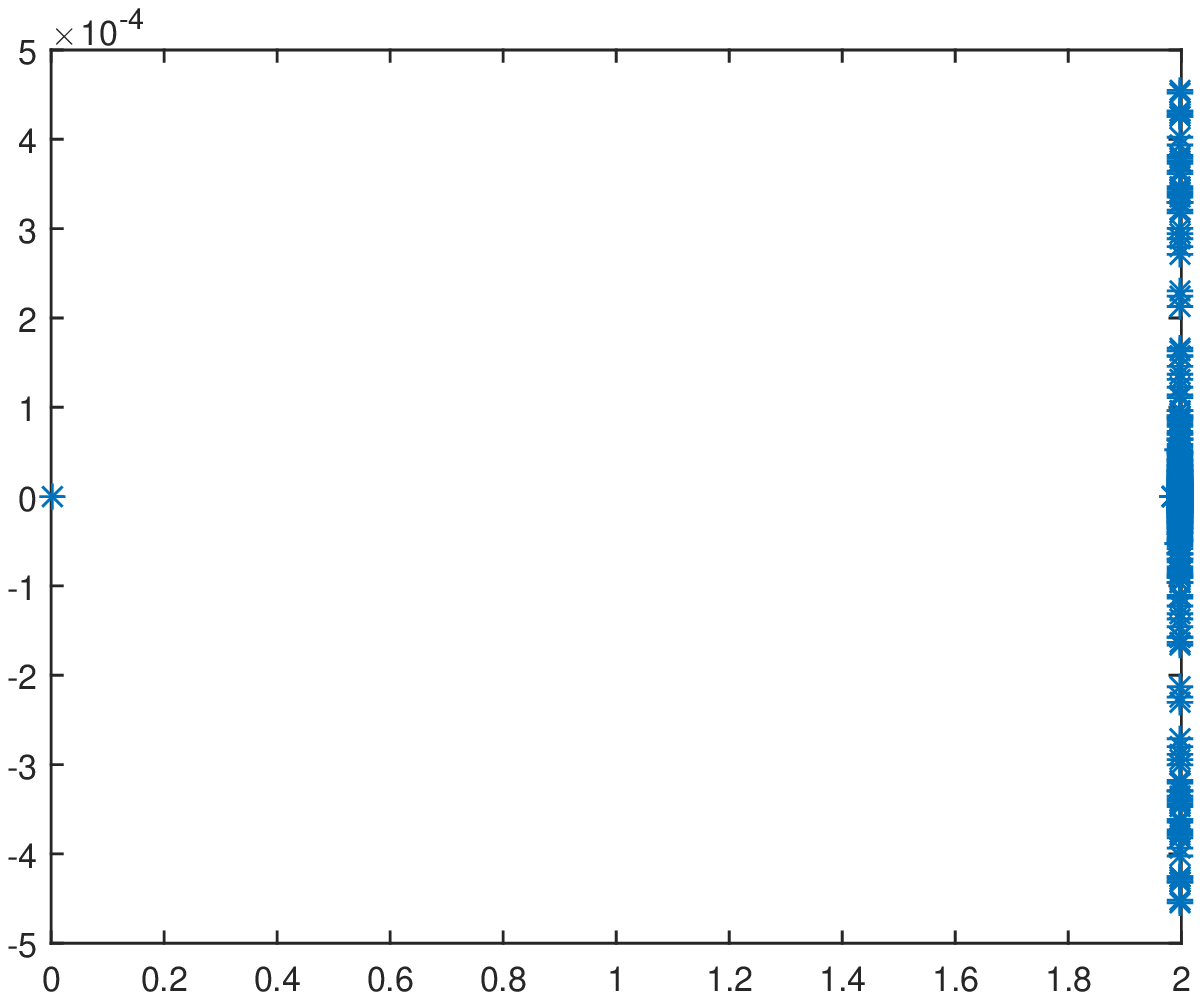}\includegraphics[height=5cm,width=5.5cm]{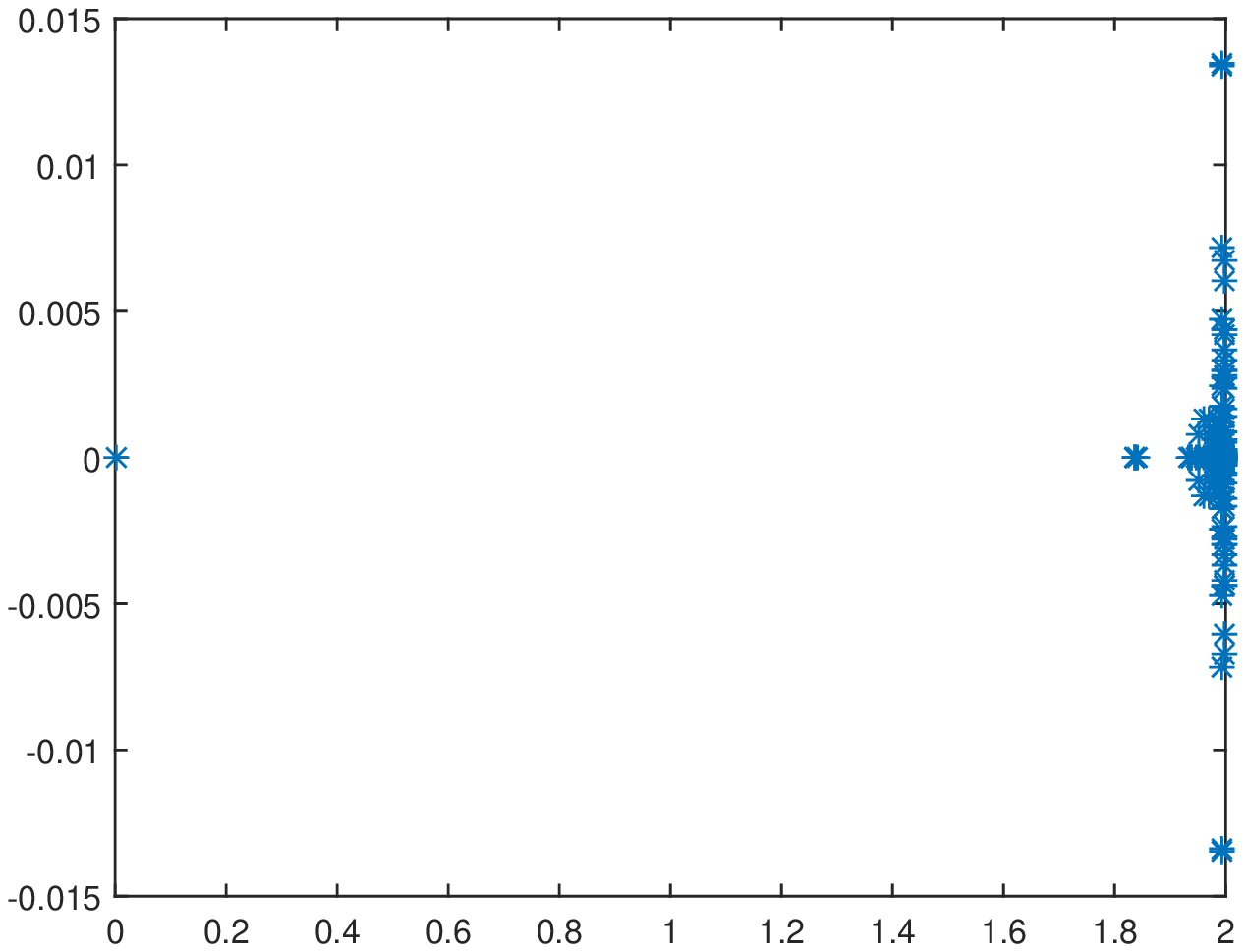}}
{\caption{The eigenvalue distribution of the saddle point matrix $\mathcal{A}$ (left) and the preconditioned matrices $\mathcal{M}^{-1}\mathcal{A}$ (middle) and $\mathcal{P}_{GSS}^{-1}\mathcal{A}$ (right) with $\alpha=10^{-4}$ and $\beta=10^{-3}$ and $\textrm{grid}=16\times 16$ for Example \ref{exp2}.}}
\end{figure}

In addition to this results, we have solved the subsystems with the coefficient matrix $H+A$ by the LU factorization and the subsystem with the coefficient matrix $S$ (step 3 of Algorithm 1) by the GMRES($5$) method.  Numerical results for different values of $\alpha$ and $\beta$
have been given in Table \ref{tab4}. As the numerical results show, for large matrices the MGSS preconditioner is superior to the GSS preconditioner. In other cases, the results are comparabele.

\begin{table}[!ht]
\begin{center}
\small
\caption{Numerical results for Example \ref{exp2} when the subsystem with the coefficient matrix $S$ is solved inexactly by GMRES($5$).}\label{tab4}
\begin{tabular}{|c|cc|ccc|ccc|ccc|} \hline
\multicolumn{3}{|c|}{}   &  \multicolumn{3}{c|}{MGSS method}  &  \multicolumn{3}{c|}{GSS method}   &  \multicolumn{3}{c|}{GMRES(5) method}  \\
  \multicolumn{3}{|c|}{}      & \multicolumn{3}{c|}{}  &  \multicolumn{3}{c|}{}  &  \multicolumn{3}{c|}{}  \\ \hline
grid            & $\alpha$ & $\beta$ &Iters & CPU  &  $R_k$    & Iters & CPU & $R_k$   & Iters & CPU & $R_k$ \\ \hline
16$\times$16      & 1e-3 & 1e-2 & 1(4) &   0.12   &  5.55e-9  & 2(3)  & 0.09 &   3.60e-8   & 78(3)& 0.09  &9.65e-8\\
                          & 1e-3 & 1e-3 & 1(4) &   0.12   &  4.41e-9  & 1(5)  & 0.07 &   3.24e-9   &&& \\
                          & 1e-3 & 1e-4 & 1(4) &   0.12   &  4.32e-9  & 1(4)  & 0.06 &   3.93e-9   &&&\\
                          & 1e-2 & 1e-3 & 2(2) &   0.13   &  4.67e-8  & 2(2)  & 0.09 &   4.85e-8   &&&\\
                          & 1e-4 & 1e-3 & 1(3) &   0.11   &  6.01e-10& 1(4)  & 0.06 &   1.65e-8   &&&\\ [1mm] \hline
32$\times$32      & 1e-3 & 1e-2 & 1(5) &   0.27   &  8.39e-9 & 3(3)  & 0.27  &   3.82e-8 & 366(5) & 0.57  & 9.99e-8 \\
                          & 1e-3 & 1e-3 & 1(5) &   0.28   &  7.81e-9 & 2(2) & 0.29  &  2.82e-8 &  &  &\\
                          & 1e-3 & 1e-4 & 1(5) &   0.27   &  7.75e-9 & 1(5) & 0.24  &  4.44e-8 &  &  &\\
                          & 1e-2 & 1e-3 & 3(1) &   0.31   &  5.81e-8  & 3(3) & 0.38  &  4.02e-8 &&&   \\
                          & 1e-4 & 1e-3 & 1(3) &   0.24   &  9.29e-9  & 1(5) & 0.22  &  8.33e-8 &&&            \\[1mm] \hline
64$\times$64      & 1e-3 & 1e-2 & 2(2)  &   2.18   &  1.82e-8 & 7(4) & 3.72   & 9.00e-8 &872(5) & 6.85  &9.97e-8\\
                          & 1e-3 & 1e-3 & 2(2)  &   2.16   &  1.78e-8 & 3(2) & 2.92 & 5.74e-8 &  &  &\\
                          & 1e-3 & 1e-4 & 2(2)  &   2.18   &  1.78e-8 & 2(3) & 3.00  &  5.36e-8  &  &  &\\
                          & 1e-2 & 1e-3 & 4(4)  &   2.32   &  7.34e-8& 7(4)  & 6.74  &  5.06e-8  &&&\\
                          & 1e-4 & 1e-3 & 1(4)  &   2.05   & 2.16e-9 & 2(4)  & 2.47 &  3.08e-8  & &&\\[1mm] \hline
 128$\times$128  & 1e-3 & 1e-2 & 2(5)  &  13.86  & 3.79e-8  & 24(4) & 60.78 & 9.35e-8 & -- &--  &--\\
                          & 1e-3 & 1e-3 & 2(5)  &  13.94  & 3.76e-8  & 8(2)   & 40.91 & 8.02e-8  &  &  &\\
                          & 1e-3 & 1e-4 & 2(5)  &  13.82  & 3.75e-8  & 3(5)   & 33.20 & 8.70e-8 &  &  &\\
                          & 1e-2 & 1e-3 & 5(4)  &  15.01  & 8.96e-8 & 25(5)  & 112.58 & 9.17e-8 & &&\\
                          & 1e-4 & 1e-3 & 1(4)  &  13.71  & 9.94e-8 & 4(2)    & 21.37  & 7.18e-8  & &&\\ \hline

\end{tabular}
\end{center}
\end{table}

\end{example}

\section{Conclusion}\label{Sec6}

We have presented a modification of the generalized shift-splitting (GSS) method say MGSS method for solving the singular saddle point problems. Semi-convergence analysis of the method as well as the eigenvalue distribution of the preconditioned matrix has been presented. The MGSS method serves the MGSS preconditioner. This preconditioner has been compared with the GSS preconditioner from the numerical point of view. Form the presented numerical results we concluded that both of the preconditioners are efficient. However, for large problems the MGSS preconditioner is superior to the GSS preconditioner from the number of iterations and the CPU time point of view.

\section*{Acknowledgment}

The authors would like to thank the three anonymous referees for their valuable comments and suggestions which substantially improved the quality of the paper.
The work of the first author is partially supported by University of Guilan.

\end{document}